\documentclass[11pt,a4paper]{amsart}
\usepackage{amsmath,amsfonts,amsthm,amsopn,color,amssymb,enumitem}
\usepackage[a4paper]{geometry}
\usepackage{palatino}
\usepackage{graphicx}
\usepackage[colorlinks=true]{hyperref}
\hypersetup{urlcolor=blue, citecolor=red, linkcolor=blue}

\usepackage{cite}
\usepackage{relsize}
\usepackage{esint}
\usepackage{verbatim}
\usepackage{mathrsfs}
\usepackage{xcolor}
\usepackage{tikz}

\newcommand{\e}{\varepsilon}

\newcommand{\ff}{{\mathtt f}}
\renewcommand{\leq}{\leqslant}
\renewcommand{\geq}{\geqslant}

\renewcommand{\d }{\delta }

\newcommand{\bd}{\boldsymbol}

\newcommand{\Ua}{{\mathcal{U}}}

\newcommand{\Ui}{{\mathcal{U}_{\delta_i, \xi_i}}}

\newcommand{\U}{{\mathcal{U}_{\delta, \xi}}}

\newcommand{\beq}{\begin{equation}}
\newcommand{\eeq}{\end{equation}}

\newtheorem{theorem}{Theorem}[section]
\newtheorem*{theorem*}{Theorem}
\newtheorem{lemma}[theorem]{Lemma}

\newtheorem{definition}[theorem]{Definition}
\newtheorem{proposition}[theorem]{Proposition}
\newtheorem{corollary}[theorem]{Corollary}
\theoremstyle{definition}
\newtheorem{remark}[theorem]{Remark}

\renewcommand{\(}{\left(}
\renewcommand{\)}{\right)}

\begin{document}

\title[Sign-changing multi-bubble solutions for the Brezis-Nirenberg problem in  four dimensions]{Sign-changing multi-bubble solutions for the Brezis-Nirenberg problem in  four dimensions}

\author[A. Pistoia]{Angela Pistoia}
\address{\noindent  Dipartimento di Scienze di Base e Applicate per l'Ingegneria, Università degli Studi di Roma Sapienza}
\email{angela.pistoia@uniroma1.it}

\author[G. M. Rago]{Giuseppe Mario Rago}
\address{\noindent  Dipartimento di Matematica, Universit\'a degli Studi di Bari Aldo Moro,Italy }
\email{g.rago6@phd.uniba.it}

\author[G. Vaira]{Giusi Vaira}
\address{\noindent  Dipartimento di Matematica, Universit\'a degli Studi di Bari Aldo Moro,Italy }
\email{giusi.vaira@uniba.it}


\begin{abstract}
We construct families of sign-changing solutions for the four-dimensional
Brezis--Nirenberg problem
\[
-\Delta u=u^3+\varepsilon u\quad\text{in }\Omega,\qquad
u=0\quad\text{on }\partial\Omega,
\]
as $\varepsilon\to0^+$.  A Lyapunov--Schmidt reduction shows that the
location and relative scales of the bubbles are governed by a signed
Green--Robin interaction matrix.  We formulate an abstract existence
criterion in terms of a simple positive eigenvalue admitting a positive
eigenvector and a stable critical set.  We then apply it to a
positive--negative pair in a general domain and to several symmetric
multi-peak configurations, including alternating regular polygons,
orthogonal polygons, one central peak surrounded by peaks of the opposite
sign, and aligned three-, four-, and five-peak patterns.  For the
two-peak solution we also prove that it has exactly two nodal domains and,
under a natural balance condition and connectedness of the boundary, that
the closure of its nodal set meets the boundary.
\end{abstract}

\maketitle

\section{Introduction and main results}

In this paper we study the four-dimensional Brezis--Nirenberg problem
\begin{equation}\label{pb}
  \begin{cases}
    -\Delta u=u^3+\varepsilon u & \text{in }\Omega,\\
    u=0 & \text{on }\partial\Omega,
  \end{cases}
\end{equation}
where $\Omega\subset\mathbb R^4$ is a smooth bounded domain and
$\varepsilon>0$ is a small parameter. Our aim is to construct
sign-changing solutions which develop several concentration points as
$\varepsilon\to0^+$, and to understand how their locations, relative
scales and signs are encoded by the geometry of the domain.

The positive Brezis--Nirenberg problem is classical. Brezis and Nirenberg
proved in \cite{BN} that, if $N\ge4$, the problem
\[
-\Delta u=u^{2^*-1}+\mu u,\qquad
u>0\ \hbox{in }\Omega,\qquad
u=0\ \hbox{on }\partial\Omega,
\]
admits a solution for every $\mu\in(0,\lambda_1)$, where
$2^*=2N/(N-2)$ and $\lambda_1$ is the first Dirichlet eigenvalue of
$-\Delta$. As the perturbation tends to zero, positive solutions may
concentrate and, after rescaling, their profile is described by the
standard bubble
$$
  \mathcal U_{\delta,\xi}(x)
  =
  \delta^{-\frac{N-2}{2}}
  U\left(\frac{x-\xi}{\delta}\right),
  \qquad
  U(x)=\alpha_N(1+|x|^2)^{-\frac{N-2}{2}},
$$
where $U$ solves
$$
-\Delta U=U^{2^*-1}\qquad\text{in }\mathbb R^N.
$$
The location of the concentration points is governed by the Green and
Robin functions; see, among others, \cite{H,FW,R}.

Since the present paper is concerned with $N=4$, throughout the paper we
write
\[
\omega:=|\mathbb S^3|=2\pi^2
\]
and use the normalization
\[
G(x,y)=\frac{1}{2\omega|x-y|^2}-H(x,y),
\]
where, for fixed $y\in\Omega$, $H(\cdot,y)$ is harmonic in $\Omega$ and
\[
H(x,y)=\frac{1}{2\omega|x-y|^2}
\qquad\text{on }\partial\Omega.
\]
The Robin function is
\[
\tau_\Omega(x):=H(x,x).
\]

For a multi-bubble configuration, the Lyapunov--Schmidt reduction leads
naturally to a Green--Robin interaction matrix. In the positive case this
matrix is
$$
M(\boldsymbol\xi)=
\begin{pmatrix}
\tau_\Omega(\xi_1)&-G(\xi_1,\xi_2)&\cdots&-G(\xi_1,\xi_k)\\
-G(\xi_1,\xi_2)&\tau_\Omega(\xi_2)&\cdots&-G(\xi_2,\xi_k)\\
\vdots&\vdots&\ddots&\vdots\\
-G(\xi_1,\xi_k)&-G(\xi_2,\xi_k)&\cdots&\tau_\Omega(\xi_k)
\end{pmatrix}.
$$
The diagonal entries represent the interaction of each bubble with the
boundary, while the off-diagonal terms describe the interaction between
distinct bubbles.

For sign-changing solutions the signs of the bubbles modify the
interaction matrix. If
\[
\boldsymbol\sigma=(\sigma_1,\ldots,\sigma_k)\in\{-1,1\}^k,
\]
we introduce the signed interaction matrix
$$
m_{ii}^{\boldsymbol\sigma}=\tau_\Omega(\xi_i),
\qquad
m_{ij}^{\boldsymbol\sigma}
=-\sigma_i\sigma_jG(\xi_i,\xi_j),
\quad i\ne j.
$$
Thus the geometry of the domain and the nodal pattern enter the reduced
problem simultaneously.\\
We now introduce the notation needed to state our results. For $k\ge1$,
let
\[
\Omega_k^*
=
\left\{
\boldsymbol\xi=(\xi_1,\ldots,\xi_k)\in\Omega^k:
\xi_i\ne\xi_j\ \text{for }i\ne j
\right\}.
\]
We shall use the notion of stable critical set introduced by Y.Y. Li in
\cite{yyl}.

\begin{definition}\label{yy1}
Let $f:D\subset\mathbb R^n\to\mathbb R$ be a smooth function. A compact
set $\mathscr K$ of critical points of $f$ is called \emph{stable} if
there exists an open neighbourhood $\Theta\Subset D$ of $\mathscr K$ such
that
\[
\deg(\nabla f,\Theta,0)\ne0.
\]
\end{definition}

Typical examples are strict local minimum or maximum sets. In particular,
if $\mathscr K$ is compact, $f$ is constant on $\mathscr K$, and
\[
f(x)<f(y)
\qquad
\text{for }x\in\mathscr K,\quad y\in\Theta\setminus\mathscr K,
\]
then $\mathscr K$ is stable. An isolated nondegenerate critical point is
another standard example.

The following abstract theorem is the basic tool used throughout the
paper.

\begin{theorem}\label{main}
Let $\boldsymbol\sigma=(\sigma_1,\ldots,\sigma_k)\in\{-1,1\}^k$ and let
$\bar\Lambda$ be a simple eigenvalue branch of
$M^{\boldsymbol\sigma}(\boldsymbol\xi)$ in a neighbourhood of a stable
critical set $\mathscr K\Subset\Omega_k^*$. Assume that
$\bar\Lambda>0$ on $\mathscr K$ and that the associated eigenvector can be
normalized as
\[
\bar{\boldsymbol\ff}(\boldsymbol\xi)
=
(1,\bar f_2(\boldsymbol\xi),\ldots,\bar f_k(\boldsymbol\xi))^T,
\qquad
\bar f_i(\boldsymbol\xi)>0.
\]
Equivalently, throughout the proof we use the normalization
\begin{equation}\label{impMbar}
M^{\boldsymbol\sigma}(\boldsymbol\xi)
\bar{\boldsymbol\ff}(\boldsymbol\xi)
=
\bar\Lambda(\boldsymbol\xi)
\bar{\boldsymbol\ff}(\boldsymbol\xi),
\qquad
\bar f_1(\boldsymbol\xi)=1.
\end{equation}
Then, for $\varepsilon>0$ sufficiently small, problem~\eqref{pb} has a
solution
\[
u_\varepsilon
=
\sum_{i=1}^k
\sigma_iP\mathcal U_{\delta_{i,\varepsilon},\xi_{i,\varepsilon}}
+\phi_\varepsilon,
\]
with
\[
\boldsymbol\xi_\varepsilon\to\boldsymbol\xi_0\in\mathscr K,
\qquad
\|\phi_\varepsilon\|_{H_0^1(\Omega)}\to0,
\]
and
\[
\frac{\delta_{i,\varepsilon}}{\delta_{1,\varepsilon}}
\longrightarrow
\bar f_i(\boldsymbol\xi_0)
\quad (i=2,\ldots,k),
\]
while
\[
\varepsilon\log\delta_{i,\varepsilon}^{-1}
\longrightarrow
8\pi^2\bar\Lambda(\boldsymbol\xi_0)
\quad (i=1,\ldots,k).
\]
If $\boldsymbol\sigma$ contains both signs, the resulting solution is
sign-changing.
\end{theorem}

The first concrete consequence holds in an arbitrary smooth bounded
domain.

\begin{theorem}[A positive--negative pair]\label{main-two-peak}
Let $\Omega\subset\mathbb R^4$ be a smooth bounded domain. For
$\varepsilon>0$ sufficiently small, problem~\eqref{pb} admits a
sign-changing solution with one positive and one negative concentration
point. More precisely, the limiting pair
$(\xi_1^*,\xi_2^*)$ belongs to the set of global minimizers in
$\Omega_2^*$ of
\[
\Lambda_2(\xi_1,\xi_2)
=
\frac{
\tau_\Omega(\xi_1)+\tau_\Omega(\xi_2)
+
\sqrt{
(\tau_\Omega(\xi_1)-\tau_\Omega(\xi_2))^2
+4G(\xi_1,\xi_2)^2
}
}{2}.
\]
The solution has exactly two nodal domains. If, in addition,
$\partial\Omega$ is connected and
\[
\tau_\Omega(\xi_1^*)=\tau_\Omega(\xi_2^*),
\]
then
\[
\overline{\{x\in\Omega:u_\varepsilon(x)=0\}}
\cap\partial\Omega
\ne\emptyset.
\]
\end{theorem}

Our second concrete result collects the symmetric multi-peak
configurations.

\begin{theorem}[Symmetric multi-peak configurations]
\label{main-configurations}
For $\varepsilon>0$ sufficiently small, problem~\eqref{pb} admits the
following sign-changing solutions.

\begin{enumerate}
\item[(a)] \textbf{Alternating regular polygons.}
In the unit ball, and in every annulus
\[
\{a<|x|<1\},\qquad 0<a<1,
\]
for every even $k\ge2$, there is a solution with $k$ alternating
concentration points at the vertices of a regular $k$-gon contained in
the $(x_1,x_2)$-plane. The solution can be chosen even in $x_3$ and
$x_4$ and anti-invariant under the rotation $R_{2\pi/k}$ in the
$(x_1,x_2)$-plane:
\[
u(R_{2\pi/k}x)=-u(x).
\]

\item[(b)] \textbf{Two orthogonal regular polygons.}
In the unit ball, for every $1\le k\le6$, there is a solution with $k$
positive concentration points at the vertices of a regular $k$-gon in
the $(x_1,x_2)$-plane and $k$ negative concentration points at the
vertices of a regular $k$-gon in the orthogonal $(x_3,x_4)$-plane.
The solution can be chosen invariant under rotations of angle $2\pi/k$
in each coordinate plane and antisymmetric under exchange of the two
planes:
\[
u(x',x'')=-u(x'',x'),
\qquad x',x''\in\mathbb R^2.
\]

\item[(c)] \textbf{One central peak and a negative polygon.}
In the unit ball, for every $k\ge2$, there is a solution with one
positive concentration point at the origin and $k$ negative
concentration points at the vertices of a regular $k$-gon in the
$(x_1,x_2)$-plane. The solution can be chosen invariant under rotation
of angle $2\pi/k$ in the $(x_1,x_2)$-plane and under the full orthogonal
group acting on the $(x_3,x_4)$-plane.

\item[(d)] \textbf{Three aligned peaks in a symmetric domain.}
Assume that $0\in\Omega$, that $\Omega$ is invariant under
\[
(x_1,x')\longmapsto(-x_1,x'),
\qquad x'\in\mathbb R^3,
\]
and under a subgroup $\mathcal G\subset O(3)$ acting on the
$x'$-variables whose common fixed-point set in $\mathbb R^4$ is the
$x_1$-axis. Assume moreover that the connected component of
$\Omega\cap(\mathbb Re_1)$ containing the origin is $(-R,R)e_1$ for
some $R>0$. Then there is a solution with one positive concentration
point at the origin and two negative concentration points at a symmetric
pair $\pm r_\varepsilon e_1$.

\item[(e)] \textbf{Four aligned peaks in a convex symmetric domain.}
Assume the symmetries in \emph{(d)} and, in addition, that $\Omega$ is
convex. Then there is a four-peak aligned solution with concentration
points
\[
-r_{2,\varepsilon}e_1,\quad
-r_{1,\varepsilon}e_1,\quad
r_{1,\varepsilon}e_1,\quad
r_{2,\varepsilon}e_1,
\qquad
0<r_{1,\varepsilon}<r_{2,\varepsilon}<R,
\]
and alternating sign pattern
\[
-,+,-,+.
\]
Equivalently, the solution can be chosen odd with respect to the
reflection $x_1\mapsto-x_1$ and invariant under the transverse
$\mathcal G$-action.

\item[(f)] \textbf{Five aligned peaks in the ball.}
In the unit ball there is a five-peak aligned solution with concentration
points
\[
0,\qquad
\pm s_\varepsilon e_1,\qquad
\pm t_\varepsilon e_1,
\qquad
0<s_\varepsilon<t_\varepsilon<1,
\]
and sign pattern
\[
+,-,+,-,+.
\]
The solution can be chosen even under $x_1\mapsto-x_1$ and invariant
under the full orthogonal group acting on the transverse variables
$(x_2,x_3,x_4)$.
\end{enumerate}

In each case the limiting geometric parameters belong to a compact stable
set of global minimizers of the corresponding symmetry-reduced positive
eigenvalue, and the relative concentration scales are determined by its
positive eigenvector.
\end{theorem}

\subsection{A comparison with higher dimensions}

A useful comparison with higher dimensions helps to explain the special
role played here by $N=4$. Sign-changing multi-bubble solutions for
critical and supercritical problems in dimensions $N\ge4$ were studied by
Micheletti and Pistoia in \cite{MiPi}. In particular, for the
Brezis--Nirenberg perturbation in dimensions $N\ge5$, their reduction leads
to a finite-dimensional functional depending simultaneously on the
concentration points and on the individual scales. In the notation used
here, its leading part has the form
\begin{equation}\label{signed-reduced-functional}
F^{\boldsymbol\sigma}(\boldsymbol d,\boldsymbol\xi)
=
a_N\left\langle
M^{\boldsymbol\sigma}(\boldsymbol\xi)\boldsymbol d,
\boldsymbol d
\right\rangle
-
b_N\sum_{j=1}^kd_j^{\frac{4}{N-2}},
\end{equation}
up to the corresponding normalization of the scale variables. The
general reduction allows one to prescribe arbitrary signs, but the
critical point analysis of \eqref{signed-reduced-functional} is rather
delicate because the location variables and all the relative scales
remain coupled. This is one of the reasons why explicit constructions in
natural geometries were obtained only for particular configurations. For
example, in the ball, \cite{MiPi} constructs a family with an even number
of bubbles arranged on a regular polygon with alternating signs; the
symmetry reduces all the scales to one parameter and all the locations to
one radial variable. The same paper also produces many sign-changing
patterns in specially constructed dumbbell-type domains. Thus the
difficulty is not the general reduction itself, but the analysis of its
finite-dimensional critical points.

The four-dimensional case has a particularly favorable structure. The
concentration scales are exponentially small in $1/\varepsilon$, and the
leading equations for the relative scales become spectral. More
precisely, once the concentration points are fixed, the scale equations
are encoded by the eigenvalue problem for
$M^{\boldsymbol\sigma}(\boldsymbol\xi)$. A simple positive eigenvalue with
an eigenvector having strictly positive components determines the
relative scales, and the remaining location problem reduces to finding a
stable critical set of that eigenvalue. In this sense the scale variables
can be eliminated through the spectral structure of the interaction
matrix. This is the main feature that allows us to analyze in dimension
four configurations which are considerably harder to detect through the
higher-dimensional reduced functional.

The first application concerns one positive and one negative bubble in an
arbitrary smooth bounded domain. In that case the favorable sign of the
interaction yields a positive simple eigenvalue with a positive
eigenvector, and its global minimum gives the desired concentration
configuration. We also study the nodal geometry of the resulting
solution: it has exactly two nodal domains and, under a natural balance
condition, the closure of the nodal set reaches the boundary.

For three or more concentration points the situation is substantially
more delicate. Nevertheless, by exploiting symmetry we construct several
families of solutions. They include alternating regular polygons in balls
and annuli, two regular polygons lying in orthogonal coordinate planes,
one positive central peak surrounded by an arbitrary number of negative
peaks, and aligned configurations with three, four and five concentration
points. Some of these patterns, and in particular the asymmetric
distribution of positive and negative peaks and the five-point aligned
configuration, appear difficult to obtain directly from the
higher-dimensional functional \eqref{signed-reduced-functional}.

This comparison strongly suggests that the configurations found here
should not be regarded as genuinely four-dimensional phenomena. We
strongly conjecture that all the sign-changing configurations constructed
in this paper persist in dimensions $N\ge5$. What seems to be special to
dimension four is not their existence, but rather the spectral mechanism
which makes their construction accessible: in higher dimensions one must
analyze the full reduced functional, with the location and scale variables
still coupled.

\subsection{Open problems and further directions}

The constructions above suggest several natural questions. We formulate
them here as problems arising directly from the limitations of the
present method.

\begin{enumerate}

\item   {\it Higher dimensions.}
We strongly conjecture that all the multi-peak configurations constructed
in Theorem~\ref{main-configurations} persist in dimensions $N\ge5$. The
main obstacle is that the scale variables no longer reduce to a spectral
problem. Instead one has to analyze the full finite-dimensional
functional \eqref{signed-reduced-functional}, where the locations and the
relative scales remain coupled. Proving this conjecture, even for the
five-point aligned configuration or for the asymmetric $1+k$ pattern,
would already be of interest.\\

\item   {\it Multi-peak nodal solutions in general domains.}
The positive--negative pair can be constructed in an arbitrary smooth
bounded domain, whereas all configurations with three or more
concentration points obtained here require substantial symmetry. Can one
find general geometric or topological conditions on $\Omega$ ensuring a
stable critical set of a positive eigenvalue of
$M^{\boldsymbol\sigma}$ for $k\ge3$?\\

\item   {\it  More aligned peaks.}
The ball admits the five-peak aligned pattern constructed in this paper.
Can one construct aligned alternating configurations with seven or more
peaks? More generally, is there a systematic spectral criterion for the
symmetry-reduced interaction matrices associated with an arbitrary odd
number of aligned peaks?\\

\item   {\it Asymmetric numbers of positive and negative
peaks outside the ball.}
The unit ball admits one positive central peak surrounded by an arbitrary
number $k$ of negative peaks. Which geometric assumptions on a general
domain allow an analogous $1+k$ construction? Can one construct families
for which the difference between the numbers of positive and negative
concentration points is larger than one without imposing radial
symmetry?\\

\item   {\it The nodal set of the two-peak solution.}
Theorem~\ref{main-two-peak} proves that the closure of the nodal set meets
the boundary under the balance condition
$
\tau_\Omega(\xi_1^*)=\tau_\Omega(\xi_2^*).
$
Is this condition necessary? Can one characterize, in terms of the Green
and Robin functions, when the nodal set meets $\partial\Omega$, and where
the contact occurs?\\

\item   {\it Selection and uniqueness of the limiting
configuration.}
Several results above use the complete compact set of global minimizers of
a reduced eigenvalue. Under what assumptions is this minimizer unique and
nondegenerate? In the annulus, if $\rho_k$ denotes a minimizing radius for
the alternating $k$-gon, we only prove
$
\operatorname{dist}(\rho_k,\{a,1\})\to0$ as $k\to\infty.
$
Which boundary component is selected, and what is the precise asymptotic
law for $\rho_k$?\\

\item   {\it A converse blow-up theory for nodal solutions.}
For positive solutions, reverse blow-up analysis leads naturally to the
Green--Robin interaction matrix. To what extent can one classify isolated
sign-changing blow-up configurations in dimension four and prove that
their limiting parameters must satisfy the spectral criticality
conditions appearing in Theorem~\ref{main}?\\

\end{enumerate}

\subsection{Organization of the paper}
The paper is organized as follows. Section~2 carries out the
Lyapunov--Schmidt reduction and proves Theorem~\ref{main}. Section~3
treats the positive--negative two-peak configuration and proves
Theorem~\ref{main-two-peak}, including the assertions on its nodal set.
Section~4 is devoted to the symmetric configurations collected in
Theorem~\ref{main-configurations}. The auxiliary matrix and Green-function
facts used in the proofs are collected in the appendices.
\\

\noindent{\bf Acknowledgments.}
The authors are deeply grateful to Professor Monica Musso for many
stimulating discussions, valuable insights, and helpful suggestions,
which greatly contributed to the development of this work.

G.~M.~Rago and G.~Vaira acknowledge support from the INdAM-GNAMPA
project ``Singolarità, nonlinearità e interazioni non locali: nuovi
approcci analitici per equazioni differenziali e sistemi''
(CUP E53C25002010001).
A.~Pistoia acknowledges support from the INdAM-GNAMPA project
``Punti critici e concentrazione di soluzioni di PDEs''
(CUP E53C25002010001).

\section{The reduction procedure}\label{sec1}

\subsection{Preliminaries}
Let $H^{1}_{0}(\Omega)$ be the Hilbert space equipped with the usual inner product and the usual norm 
\begin{equation*}
\langle u,v \rangle := \int_{\Omega} \nabla u \cdot \nabla v ,\quad \| u \|  := \left( \int_{\Omega} | \nabla u |^{2} \right)^{\frac{1}{2}}.
\end{equation*}
For $r \in [1, +\infty)$ the space $L^{r}(\Omega)$ is  equipped with the standard norm 
\begin{equation*}
\| u \|_{r} = \left( \int_{\Omega} |u|^{r} \right)^{\frac{1}{r}}.
\end{equation*}

As usual, let $i^{*}_\Omega: L^{\frac{4}{3}}(\Omega) \to H^{1}_{0}(\Omega)$ be the adjoint operator of the embedding $i_\Omega : H^{1}_{0}(\Omega) \to L^{4}(\Omega)$, i.e. 
$$u=i^{*}_\Omega(f)\ \hbox{if and only if}\ 
\langle u, \varphi \rangle = \int_{\Omega} f(x) \varphi(x) \ dx
\ \hbox{
for all}\ \varphi \in H^{1}_{0}(\Omega).
$$
The operator $i^{*}_\Omega : L^{\frac{4}{3}}(\Omega) \to H^{1}_{0}(\Omega)$ is continuous as 
\begin{equation*}
\| i^{*}_\Omega(f) \|_{H^{1}_{0}(\Omega)} \leq S^{-1} \| f \|_{\frac{4}{3}}
\end{equation*}
where $S$ is the best constant for Sobolev embedding.\\
Therefore, the  problem \eqref{pb} can be rewritten as 
\begin{equation}\label{213}
u=i^{*}_\Omega(f(u)+ \varepsilon u),\ u \in H^{1}_{0}(\Omega),
\end{equation}
where $f(u)=|u|^2 \cdot u.$ \\

Given a point $\xi\in\Omega,$ we denote by $P\mathcal{U}_{\delta,\xi}$ the projection of $\mathcal{U}_{\delta,\xi}$ into $H^1_0(\Omega)$, i.e. the unique solution  of 
\begin{eqnarray}\label{Pre0001}
 -\Delta\, P\mathcal{U}_{\delta,\xi}=\mathcal{U}_{\delta,\xi}^{3}\ \hbox{in}\ \Omega,\ 
 P\mathcal{U}_{\delta,\xi}=0\ \hbox{on}\ \partial\Omega.
\end{eqnarray}
\\  It is well known that the following expansion holds (see \cite{R})
$$
P\U:=\U-\mathfrak C\d H(x, \xi)+\mathcal O\left(\d^{3}\right)\quad \mbox{as}\, \d\to 0$$
uniformly with respect to $x\in\Omega$ and $\xi$ in compact sets of $\Omega$ and 
\begin{equation}\label{exp1}
P\U:=\mathfrak C\d G(x, \xi)+\mathcal O\left(\d^{3}\right)\quad \mbox{as}\, \d\to 0\end{equation}
uniformly with respect to $x$ in compact sets of $\Omega\setminus\{\xi\}$ and $\xi$ in compact sets of $\Omega$ where $\mathfrak C:=2\alpha_4\omega=8\sqrt2\pi^2$.\\  

\subsection{The ansatz} 
Let $k \geq 1$ be a fixed integer. We look for a solution to \eqref{213} of the form 
$$
u_\e = \sum_{i=1}^{k} \sigma_{i} P\Ui+\phi_\e, 
$$
where $\sigma_i\in\{-1,1\}$ and the concentration scales $\delta_i=\delta_i(\varepsilon)$ are chosen as
\begin{equation}\label{delta}
\delta_1:=e^{-\frac{8\pi^2 \lambda}{ \e}},\ \d_i:=e^{-\frac{8\pi^2 \lambda}{ \e}}d_i\ \hbox{with}\ \lambda>0\ \hbox{and}\
 d_i>0\ \hbox{for}\  i=2, \ldots, k
\end{equation}  
and the concentration points are $\xi_{i}$ belong to the set
\begin{equation*}
D_\rho = \{ \pmb{\xi} \in \Omega^{*}_k\,\,:\,\, \mbox{dist}(\xi_{i}, \partial \Omega) \geq 2\rho, |\xi_{i} - \xi_{j}| \geq 2\rho,\,\, \forall i,j = 1, \cdots, k, i \neq j \}.
\end{equation*} for some $\rho>0$ small.\
In the following, we agree that $\boldsymbol \delta:=(\delta_1,\dots,\delta_k)$ and $\boldsymbol \xi:=(\xi_1,\dots,\xi_k).$\\
The higher order term  $\phi_\e$ belongs to the space  
\begin{equation*}
\mathcal K^{\perp}_{\pmb{\d}, \pmb{\xi}} =\left \{ \phi \in H^{1}_{0}(\Omega)\,\, :\,\,\ \langle \phi, \psi \rangle = 0\ \forall\ \psi\in \mathcal K_{\pmb{\d}, \pmb{\xi}}\right\},
\end{equation*}
where
\begin{equation*}
\mathcal K_{\pmb{\d}, \pmb{\xi}} = \mbox{span} \{ P \psi_{\delta_{i}, \xi_{i}}^{j}\,\, :\,\, i= 1, \cdots, k, j=0, \cdots, 4\}.
\end{equation*}
Here $P\psi_{\delta_{i}, \xi_{i}}^{j}$ are the orthogonal projections onto $H^1_0(\Omega)$ of the functions
\begin{equation*}
\psi_{\delta_{i}, \xi_{i}}^{j}(x) = \frac{1}{\delta_{i}} \psi^{j} \left(\frac{x-\xi_{i}}{\delta_{i}} \right),\   j=0, \ldots, 4,\ i=1, \ldots, k,
\end{equation*}
where 
\begin{equation*}
\psi^{0}(x) := \Ua(x) + \nabla \Ua(x) \cdot x  = \alpha_4 \frac{|x|^{2}-1}{(1+|x|^{2})^{2}}
\end{equation*}
and 
\begin{equation*}
\psi^{j}(x) : = \frac{\partial \Ua}{\partial x_{j}}(x) = -2\alpha_4 \frac{x_{j}}{(1+|x|^{2})^{2}}, \  j=1, \cdots, 4.
\end{equation*}
generate the space of solutions of the linear equation
\begin{equation*}
-\Delta \psi = 3\Ua^{2} \psi \ \hbox{in} \   \mathbb{R}^{4}.
\end{equation*}
It is useful to recall the well-known expansions
$$
P \psi^{0}_{\delta_{i}, \xi_{i}} =\psi^{0}_{\delta_{i}, \xi_{i}} -\mathfrak C\delta_i H(x, \xi_i)+\mathcal O(\delta_i^2)$$
and for all $j=1, \ldots, 4$
$$
P \psi^{j}_{\delta_{i}, \xi_{i}} =\psi^{j}_{\delta_{i}, \xi_{i}} - \mathfrak C^2\delta_i^2\partial_{\xi_j} H(x, \xi_i)+\mathcal O(\delta_i^3).$$
\subsection{An equivalent system}
Let $\Pi_{\boldsymbol\delta,\boldsymbol\xi}$ and $\Pi^\perp_{\boldsymbol\delta,\boldsymbol\xi}$ denote the orthogonal projections of $H_0^1(\Omega)$ onto $\mathcal K_{\boldsymbol\delta,\boldsymbol\xi}$ and $\mathcal K^\perp_{\boldsymbol\delta,\boldsymbol\xi}$, respectively.
In order to simplify the notations we let also $W_{\pmb{\d}, \pmb{\xi}}:=\sum_{i=1}^k \sigma_iP\Ui$. \\ Since we seek a solution of the form $W_{\pmb{\d}, \pmb{\xi}}+\phi_\e$ then
equation \eqref{213} can be rewritten as the following system of two equations 
\begin{equation}\label{10}
\Pi^{\perp}_{\pmb{\d}, \pmb{\xi}} [ \mathcal{L}_{\pmb{\d}, \pmb{\xi}}(\phi_\e) - \mathcal{E}_{\pmb{\d}, \pmb{\xi}} - \mathcal{N}_{\pmb{\d}, \pmb{\xi}}(\phi_\e)] = 0
\end{equation}
and
\begin{equation}\label{11}
\Pi_{\pmb{\d}, \pmb{\xi}} [\mathcal{L}_{\pmb{\d}, \pmb{\xi}}(\phi_\e) - \mathcal{E}_{\pmb{\d}, \pmb{\xi}} - \mathcal{N}_{\pmb{\d}, \pmb{\xi}}(\phi_\e)] = 0,
\end{equation}
where the linear operator $\mathcal{L}_{\pmb{\d}, \pmb{\xi}}$ is 
$$
\mathcal{L}_{\pmb{\d}, \pmb{\xi}}(\phi_\e) = \phi_\e - i^{*}_\Omega(3\phi_\e W_{\pmb{\d}, \pmb{\xi}}^{2} + \varepsilon \phi_\e),
$$
the error term $\mathcal{E}_{\pmb{\d}, \pmb{\xi}}$ is 
$$
\mathcal{E}_{\pmb{\d}, \pmb{\xi}}= i^{*}_\Omega(W_{\pmb{\d}, \pmb{\xi}}^{3} + \varepsilon W_{\pmb{\d}, \pmb{\xi}}) - W_{\pmb{\d}, \pmb{\xi}}
$$
and the nonlinear term $\mathcal{N}_{\pmb{\d}, \pmb{\xi}}$ is 
$$
\mathcal{N}_{\pmb{\d}, \pmb{\xi}}(\phi_\e) = i^{*}_\Omega ( \phi_{\e}^{3} + 3 \phi_{\e}^{2}W_{\pmb{\d}, \pmb{\xi}}) .
$$

\subsection{Solving equation (\ref{10})}
The solvability of \eqref{10} in terms of $(\boldsymbol\delta, \bd\xi)$ is the first step in the Ljapunov-Schmidt procedure and follows by the proposition below.
\begin{proposition}\label{fixedpoint}
For any $\rho>0$ there exist $C>0$ and $\varepsilon_{0}>0$ such that for any $\varepsilon, \d_i \in (0, \varepsilon_{0})$ and any $\bd\xi\in D_\rho$ there exists a unique $\phi=\phi_{\pmb{\d}, \pmb{\xi}}\in \mathcal K^{\perp}_{\pmb{\d}, \pmb{\xi}}$ solving \eqref{10} and satisfying 
\begin{equation}\label{stimaphi}
\| \phi \| \leq C\( |\pmb{\d}|^2+\e|\pmb{\d}|\).\end{equation}
\end{proposition}
\begin{proof} The proof is standard and we refer to  \cite{MP, PRV, PR} for the details. 
\end{proof}

\subsection{Solving equation (\ref{11})}
Let $u_\e=W_{\pmb{\d}, \pmb{\xi}}+\phi$ where $\phi\in \mathcal K^\bot_{\pmb{\d}, \pmb{\xi}}$ is the function found in 
Proposition \ref{fixedpoint}.   Equation \eqref{11} rewrites as
\begin{equation}\label{r1}
-\Delta u_\e-f(u_\e)-\e u_\e
=\sum_{h=1}^k\sum_{j=0}^4 c_h^j\,f'(\mathcal U_{\delta_h,\xi_h})\psi^j_{\delta_h,\xi_h}
\end{equation}
for some real numbers $c_h^j$ that depend on $\boldsymbol\delta$ and   $\boldsymbol \xi$. 
The second step  in the Ljapunov-Schmidt procedure
 consists in  finding  the rate parameters $\boldsymbol\delta$ and the points $\boldsymbol \xi$ so that all the $c_h^j$'s  in \eqref{r1} are zero and this is done by  solving a reduced problem as stated  in the Proposition below. We omit the proof since it can be obtained as in \cite{PRV}. 
 
 \begin{proposition}\label{propro}
Up to nonzero dimensional constants, as $\e\to0$ one has
\[
c_h^0=
\delta_h^2\tau_\Omega(\xi_h)
-\sum_{j\neq h}\sigma_j\sigma_h\delta_j\delta_hG(\xi_j,\xi_h)
+\frac1{4\omega}\e\delta_h^2\log\delta_h
+o(|\boldsymbol\delta|^2),
\]
for $h=1,\ldots,k$, and
\[
c_h^\ell=
\frac{\partial}{\partial(\xi_h)_\ell}
\left[
\delta_h^2\tau_\Omega(\xi_h)
-2\sum_{j\neq h}\sigma_j\sigma_h\delta_j\delta_hG(\xi_j,\xi_h)
\right]
+o(|\boldsymbol\delta|^2),
\]
for $h=1,\ldots,k$ and $\ell=1,\ldots,4$. The remainders are uniform with respect to $(\boldsymbol\xi,\mathbf d,\lambda)$ in compact subsets of $D_\rho\times(0,+\infty)^{k-1}\times(0,+\infty)$.
\end{proposition}

By Proposition \ref{fixedpoint} and  Proposition \ref{propro}, it follows that the problem reduces to find $\delta_h>0$ and $\xi_h\in\Omega$ for $h=1, \ldots, k$ such that 
$$\left\{\begin{aligned} &\left(\d_h \tau_\Omega(\xi_h)-\sum_{j\neq h}\sigma_h\sigma_j\d_j G(\xi_h, \xi_j)+\frac{1}{8\pi^2}\e\d_h \ln\d_h\right)\left(1+o(1)\right)=0\, \quad &h=1, \ldots, k\\\\
&\left(\d_h\frac{\partial}{\partial(\xi_h)_\ell} \tau_\Omega(\xi_h)-2\sum_{j\neq h}\sigma_h\sigma_j\d_j \frac{\partial}{\partial(\xi_h)_\ell}G(\xi_h, \xi_j)\right)\left(1+o(1)\right)=0\quad &h=1, \ldots, k.\end{aligned}\right.$$
More precisely, taking into account the choice of the $\delta_h$'s in \eqref{delta}, we have to find $\lambda>0,$ $d_2,\dots,d_{k}\in (0,+\infty)$ 
and $(\xi_1,\dots,\xi_k)\in D_\rho$ solution of the  system 
\begin{equation}\label{sys2}\left\{\begin{aligned} &\tau_\Omega(\xi_1)-\sum_{j=2}^k \sigma_1\sigma_jd_j G(\xi_1, \xi_j)-\lambda+o(1)=0\,\\
&d_h \tau_\Omega(\xi_h)-\sigma_1\sigma_hG(\xi_1, \xi_h)-\sum_{j=2\atop j\neq h}^k\sigma_h\sigma_j d_j G(\xi_h, \xi_j)-\lambda d_h+o(1)=0,\ h=2, \ldots, k\\
&\frac{\partial}{\partial(\xi_1)_\ell}\tau_\Omega(\xi_1) -2\sum_{j=2}^k \sigma_1\sigma_jd_j\frac{\partial}{\partial(\xi_1)_\ell} G(\xi_1, \xi_j)+o(1)=0,\ \ell=1, \ldots, 4\\
&d_h \frac{\partial}{\partial(\xi_h)_\ell} \tau_\Omega(\xi_h)-2\sigma_1\sigma_h\frac{\partial}{\partial(\xi_h)_\ell}G(\xi_1, \xi_h)-2\sum_{j=2\atop j\neq h}\sigma_j\sigma_hd_j \frac{\partial}{\partial(\xi_h)_\ell} G(\xi_h, \xi_j)+o(1)=0,\\ &\hskip9truecm h=2, \ldots, k,\   \ell=1, \ldots, 4.\end{aligned}\right.\end{equation}

\subsection{Proof of Theorem \ref{main}}
We aim to solve system \eqref{sys2}. We set
$\mathbf{d}:=(d_2,\ldots,d_k)$ and we  introduce three different functions.
The first one $F_1: D_\rho \times (0, +\infty)^{k-1}\times (0, +\infty)\to\mathbb R$
is related to the first equation in \eqref{sys2}:
$$ F_1(\boldsymbol\xi, \mathbf{d}, \lambda)= \tau_\Omega(\xi_1)-\sum_{j=2}^k \sigma_1\sigma_jd_j G(\xi_1, \xi_j)-\lambda.$$
The second one  $F_2:D_\rho \times (0, +\infty)^{k-1}\times (0, +\infty)\to\mathbb R^{k-1}$
is related to the second $k-1$ equations in \eqref{sys2}:
$$F_2(\boldsymbol\xi, \mathbf{d}, \lambda)=\left(\overline M^{\boldsymbol\sigma}(\boldsymbol\xi)-\lambda \mathtt{Id}\right)\mathbf{d}^T-\overline G^{\boldsymbol\sigma}(\boldsymbol\xi)$$
where $$\overline M^{\boldsymbol\sigma}(\boldsymbol\xi) :=\left(\begin{matrix} \tau_\Omega(\xi_2) & -\sigma_2\sigma_3G(\xi_2, \xi_3) & \ldots & -\sigma_2\sigma_kG(\xi_2, \xi_k)\\
-\sigma_2\sigma_3G(\xi_2, \xi_3) & \tau_\Omega(\xi_3) &\ldots &-\sigma_3\sigma_kG(\xi_3, \xi_k)\\
\ldots &\ldots &\ldots &\ldots\\
\ldots &\ldots &\ldots &\ldots\\
-\sigma_2\sigma_kG(\xi_2, \xi_k) & -\sigma_3\sigma_kG(\xi_3, \xi_k) & \ldots &\tau_\Omega(\xi_k)
\end{matrix}\right)\ \hbox{and}\ \overline G^{\boldsymbol\sigma}(\boldsymbol\xi):=\left(\begin{matrix} \sigma_1\sigma_2G(\xi_1, \xi_2)\\ \sigma_1\sigma_3G(\xi_1, \xi_3) \\
\ldots \\
\ldots \\
\sigma_1\sigma_kG(\xi_1, \xi_k) 
\end{matrix}\right)$$
The third one $F_3:D_\rho \times (0,+\infty)^{k-1}\times(0,+\infty)\to\mathbb R^{4k}$ is related to the last $4k$ equations in \eqref{sys2}. Setting $\widehat{\mathbf d}:=(1,d_2,\ldots,d_k)^T$, we define
$$F_3(\boldsymbol\xi,\mathbf d,\lambda):=\left(\widetilde M^{1,\boldsymbol\sigma}(\boldsymbol\xi)\widehat{\mathbf d},\ldots,\widetilde M^{4,\boldsymbol\sigma}(\boldsymbol\xi)\widehat{\mathbf d}\right).$$
where
$$\widetilde M^{\ell, \boldsymbol\sigma}(\boldsymbol\xi)=\(\widetilde m_{ij}^{\ell, \boldsymbol\sigma}(\xi)\)_{1\leq i, j\leq k}\ \hbox{and}\ \widetilde m_{ij}^{\ell, \boldsymbol\sigma}(\xi):=\left\{\begin{aligned}&\frac{\partial}{\partial(\xi_i)_\ell}  \tau_\Omega(\xi_i)\quad \ \ \ & \mbox{if}\,\, i=j\\
&-2\sigma_i\sigma_j\frac{\partial}{\partial(\xi_i)_\ell}  G(\xi_i, \xi_j)\quad &\mbox{if}\,\, i\neq j.\\
\end{aligned}\right.$$
Finally we also define
$$F(\boldsymbol\xi, \mathbf{d}, \lambda) :=(F_1(\boldsymbol\xi, \mathbf{d}, \lambda), F_2(\boldsymbol\xi, \mathbf{d}, \lambda), F_3(\boldsymbol\xi, \mathbf{d}, \lambda) ).$$
It is clear that  solving \eqref{sys2} is equivalent to finding   $\boldsymbol\xi=\boldsymbol\xi(\varepsilon) \in D_\rho$,  $\lambda=\lambda(\varepsilon)\in (0, +\infty)$ and  $\mathbf{d}=\mathbf d(\varepsilon)=(d_2, \ldots, d_k)\in (0, +\infty)^{k-1}$ so that
\begin{equation}\label{globale}F(\boldsymbol\xi, \mathbf{d}, \lambda ) +o(1)=0,\end{equation}
where $o(1)$ is a continuous function in $(\boldsymbol\xi, \mathbf{d}, \lambda)$ which converges toward zero as $\varepsilon\to0$ 
uniformly  in compact sets of $D_\rho \times (0, +\infty)^{k-1}\times (0, +\infty).$
\\
We shall use  Lemma \ref{isolated} to prove that if $\mathscr K$ is a stable critical set of $\bar\Lambda(\boldsymbol \xi)$ in the sense of Definition \ref{yy1} then there exists a set $\mathscr Z$ such that
$F(\boldsymbol\xi,\mathbf{d},\lambda)=0$ for any $(\boldsymbol\xi,\mathbf{d},\lambda)\in \mathscr Z$ and 
 $\mathtt{deg}\left(F,\Xi,0\right)\not=0$ for an open neighbourhood of $\mathscr Z$.  By the properties of Brouwer degree, we immediately deduce the existence as $\varepsilon$ is small enough of $(\boldsymbol\xi (\varepsilon),\mathbf{d} (\varepsilon),\lambda (\varepsilon))$  close to $\mathscr Z$
solution of \eqref{globale}.\\
First of all, we observe that for any $\boldsymbol\xi\in D_\rho$  there exists a unique $\mathbf{d}=\mathbf{d}(\xi)\in(0, +\infty)^{k-1}$ and
$\lambda=\lambda(\xi)\in (0,+\infty)$ such that 
$$F_1(\boldsymbol\xi, \mathbf{d}(\boldsymbol\xi), \lambda(\boldsymbol\xi))=0\ \hbox{and}\ F_2(\boldsymbol\xi, \mathbf{d}(\boldsymbol\xi), \lambda(\boldsymbol\xi))=0.$$
Indeed, those equations are equivalent to the system
$$\left\{\begin{aligned}& \tau_\Omega(\xi_1)-\sum_{j=2}^k \sigma_1\sigma_jd_j G(\xi_1, \xi_j)-\lambda=0\\
 &\left(\overline M^{\bd\sigma}(\boldsymbol\xi)-\lambda \mathtt {Id}\right)\mathbf{d}^T-\overline G^{\bd\sigma}(\boldsymbol\xi)=0.\end{aligned}\right.$$
A direct computation shows that it has the solution
$$\lambda(\xi):=\bar\Lambda( \boldsymbol\xi)\quad \hbox{and}\quad  \mathbf d(\xi)^T:=\left( \overline M^{\bd\sigma}(\boldsymbol\xi)-\bar\Lambda( \boldsymbol\xi) \mathtt {Id} \right)^{-1}\overline G^{\bd\sigma}(\boldsymbol\xi).$$
Indeed by the first equation we deduce
$$\lambda=\tau_\Omega(\xi_1)-\sum_{j=2}^k \sigma_j d_j G(\xi_1, \xi_j)$$
and combining the two equations we get
\begin{equation}\label{1}M^{\bd\sigma}(\boldsymbol\xi)(1, \mathbf{d})^T=\lambda (1, \mathbf{d})^T.\end{equation}
Now if we multiply \eqref{1} by $\bar{\ff}(\boldsymbol\xi)$ given in \eqref{impMbar} we get
$$\bar\Lambda( \boldsymbol\xi)\langle \ff(\boldsymbol\xi)^T, (1, \mathbf{d})^T\rangle=\langle \ff^T(\boldsymbol\xi), M^{\bd\sigma}(\boldsymbol\xi)(1, \mathbf{d})^T\rangle =\lambda\langle \ff(\boldsymbol\xi)^T, (1, \mathbf{d})^T\rangle,$$  which implies that
$\lambda=\bar\Lambda( \boldsymbol\xi).$ Moreover,  since $\bar\Lambda(\bd\xi)$ is simple then the  matrix $\overline M^{\bd\sigma}(\boldsymbol\xi)-\bar\Lambda( \boldsymbol\xi) \mathtt {Id}$ is invertible (see \cite{PRV} for other details).
At the end, we can show that (reasoning as in \cite{PRV})

\begin{equation}\label{var2}
\nabla_{\xi_h}\bar\Lambda(\boldsymbol\xi)
=
\frac{2\,\bar f_h(\boldsymbol\xi)}
{|\bar{\boldsymbol\ff}(\boldsymbol\xi)|^2}\,
F_{3,h}(\boldsymbol\xi,\mathbf d(\boldsymbol\xi),\bar\Lambda(\boldsymbol\xi)),
\qquad h=1,\ldots,k,
\end{equation}
where $F_{3,h}\in\mathbb R^4$ denotes the block of $F_3$ corresponding to the variable $\xi_h$.
Thus $F_3=0$ if and only if $\nabla\bar\Lambda=0$. Since every
$\bar f_h$ is positive, the two vector fields differ by an invertible positive diagonal
matrix. Consequently they have the same local Brouwer degree (after shrinking the
neighbourhood if necessary).
Since $\mathscr K$ is a stable critical set of $\bar\Lambda$, Definition~\ref{yy1} and \eqref{var2} imply that the reduced location field $F_3(\boldsymbol\xi,\mathbf d(\boldsymbol\xi),\bar\Lambda(\boldsymbol\xi))$ has nonzero degree in a sufficiently small neighbourhood $\Theta$ of $\mathscr K$. Therefore, by Lemma~\ref{isolated}, the set
$\mathscr Z:=\left\{(\boldsymbol\xi,\mathbf d(\boldsymbol\xi),\bar\Lambda(\boldsymbol\xi)):\boldsymbol\xi\in\mathscr K\right\}$ is such that 
$\mathtt {deg}(F, \Xi, 0)\not=0,$  for some neighbourhood $\Xi$ of $\mathscr Z.$ This completes the proof.
\section{A positive--negative pair in a general domain}
We now consider the simplest genuinely nodal configuration: one positive and one negative concentration point. We seek a solution of \eqref{pb} of the form $$u_\e:=P\Ua_{\delta_1, \xi_1}-P\Ua_{\delta_2, \xi_2}+\phi_\e.$$ 

We let the matrix $M^{\boldsymbol\sigma}(\boldsymbol\xi)$ defined as \beq\label{Msigma2}M^{\boldsymbol\sigma}(\boldsymbol\xi) :=\left(\begin{matrix} \tau_\Omega(\xi_1) & G(\xi_1, \xi_2) \\
G(\xi_1, \xi_2) & \tau_\Omega(\xi_2) 
\end{matrix}\right)\eeq The matrix \eqref{Msigma2} has two simple real eigenvalues which are 
$$\Lambda_1(\bd\xi):=\frac{\tau_\Omega(\xi_1) +\tau_\Omega(\xi_2)-\sqrt{(\tau_\Omega(\xi_1) -\tau_\Omega(\xi_2))^2+4 (G(\xi_1, \xi_2))^2 } }{2}$$ and $$\Lambda_2(\bd\xi):=\frac{\tau_\Omega(\xi_1) +\tau_\Omega(\xi_2)+\sqrt{(\tau_\Omega(\xi_1) -\tau_\Omega(\xi_2))^2+4 (G(\xi_1, \xi_2))^2 } }{2}.$$
The corresponding eigenvectors are
$$\bd\ff_1(\bd\xi):=\left(\begin{matrix} 1\\ \ff_{1, 2}(\bd\xi)\end{matrix}\right),\qquad \bd\ff_2(\bd\xi):=\left(\begin{matrix} 1\\ \ff_{2, 2}(\bd\xi)\end{matrix}\right),$$
where $$\ff_{1, 2}(\bd\xi):=\frac{\tau_\Omega(\xi_2)-\tau_\Omega(\xi_1)-\sqrt{(\tau_\Omega(\xi_2)-\tau_\Omega(\xi_1))^2+4(G(\xi_1, \xi_2))^2}}{2G(\xi_1, \xi_2)}$$ and $$\ff_{2, 2}(\bd\xi):=\frac{\tau_\Omega(\xi_2)-\tau_\Omega(\xi_1)+\sqrt{(\tau_\Omega(\xi_2)-\tau_\Omega(\xi_1))^2+4(G(\xi_1, \xi_2))^2}}{2G(\xi_1, \xi_2)}.$$ We remark that $\ff_{1, 2}(\bd\xi)<0$ while $\ff_{2, 2}(\bd\xi)>0$. \\ We are able to show Theorem~\ref{main-two-peak}. 

\begin{proof}[Proof of Theorem~\ref{main-two-peak}]
\begin{itemize}
\item[(i)]
We apply Theorem~\ref{main} with $\boldsymbol\sigma=(1,-1)$. The eigenvalue
$\Lambda_2(\boldsymbol\xi)$ is simple and strictly positive on
$\Omega_2^*$, and its eigenvector $\boldsymbol\ff_2(\boldsymbol\xi)$ has
strictly positive components.

Moreover,
\[
\Lambda_2(\boldsymbol\xi)\longrightarrow+\infty
\qquad\text{as }\boldsymbol\xi\to\partial\Omega_2^*.
\]
Indeed, if one of the points approaches $\partial\Omega$, the corresponding
Robin function diverges to $+\infty$; if $\xi_1-\xi_2\to0$, then
$G(\xi_1,\xi_2)\to+\infty$. Hence
\[
\mathscr K_2:=
\left\{\boldsymbol\xi\in\Omega_2^*:
\Lambda_2(\boldsymbol\xi)=\min_{\Omega_2^*}\Lambda_2\right\}
\]
is a nonempty compact set. Since $\Lambda_2$ is strictly larger than its
minimum outside $\mathscr K_2$, the set $\mathscr K_2$ is a strict local
minimum set and therefore a stable critical set in the sense of
Definition~\ref{yy1}. Theorem~\ref{main} yields the required family, with
$\boldsymbol\xi_\varepsilon\to\boldsymbol\xi^*\in\mathscr K_2$.
\item[(ii)] Let
\begin{equation}\label{4.1}
u_\varepsilon(x)
=
P \mathcal U_{\delta_{1,\varepsilon},\xi_{1,\varepsilon}}(x)
-
P \mathcal U_{\delta_{2,\varepsilon},\xi_{2,\varepsilon}}(x)
+
\phi_\varepsilon(x)
\end{equation}
be a solution to problem \eqref{pb} such that $\phi_\varepsilon$ satisfies \eqref{stimaphi},
$$\delta_{1,\varepsilon}=e^{-\frac{8\pi^2}{\e}\Lambda_2(\xi_{1, \e}, \xi_{2, \e})},\quad \delta_{2,\varepsilon}=e^{-\frac{8\pi^2}{\e}\Lambda_2(\xi_{1, \e}, \xi_{2, \e})}\ff_{2, 2}(\bd\xi_\e),
$$
where $\xi_{i,\varepsilon}\to\xi_i^*$ and $\ff_{2, 2}(\bd\xi_\e)\to \ff_{2, 2}(\bd\xi^*)$ as $\e\to 0$. We claim that the set
\[
\Omega\setminus\{x\in\Omega:u_\varepsilon(x)=0\}
\]
has exactly two connected components. Indeed, the function
\[
u_{1,\varepsilon}(y)
:=
\delta_{1,\varepsilon}u_\varepsilon(\delta_{1,\varepsilon}y+\xi_{1,\varepsilon}),
\qquad
y\in\Omega_{1,\varepsilon}
:=
\frac{1}{\delta_{1,\varepsilon}}
(\Omega-\xi_{1,\varepsilon})
\]
solves
\[
\begin{cases}
-\Delta u_{1,\varepsilon}
= (|u_{1,\varepsilon}|^{2}+\e\delta_{1, \e}^2 ) u_{1,\varepsilon}
& \text{in }\Omega_{1,\varepsilon},\\
u_{1,\varepsilon}=0
& \text{on }\partial\Omega_{1,\varepsilon}.
\end{cases}
\]
Since, as $\varepsilon\to0$,
$\Omega_{1,\varepsilon}\to\mathbb R^4$
and
$u_{1,\varepsilon}\to \mathcal U_{1,0}$ in $D^{1,2}(\mathbb R^4)$,
by standard regularity theory it follows that
\[
u_{1,\varepsilon}\to \mathcal  U_{1,0}
\quad\text{in }C^1_{\mathrm{loc}}(\mathbb R^N).
\]
In particular, for every fixed $r>0$ there exists $\varepsilon_0>0$
such that for every $\varepsilon\in(0,\varepsilon_0)$,
\[
u_{1,\varepsilon}(y)>0
\qquad\forall y\in B(0,r),
\]
or equivalently
\begin{equation}\label{4.2}
u_\varepsilon(x)>0,
\qquad
x\in B(\xi_{1,\varepsilon},r\delta_{1,\varepsilon}).
\end{equation}
Analogously,
\begin{equation}\label{4.3}
u_\varepsilon(x)<0,
\qquad
x\in B(\xi_{2,\varepsilon},r\delta_{2,\varepsilon}).
\end{equation}
Then, $u_\e(x)$ has at least two nodal domains, namely
$\Omega\setminus\{u_\varepsilon=0\}$
has at least two connected components
$\Omega_\varepsilon^+\subset\{u_\varepsilon>0\}$
and
$\Omega_\varepsilon^-\subset\{u_\varepsilon<0\}$.
Moreover, by \eqref{4.2} and \eqref{4.3} it follows that
\[
B(\xi_{1,\varepsilon},r\delta_{1,\varepsilon})
\subset
\Omega_\varepsilon^+,
\]
and
\[
B(\xi_{2,\varepsilon},r\delta_{2,\varepsilon})
\subset
\Omega_\varepsilon^-.
\]
Next we show that $u_\e$ has not more than two nodal domains for $\e$ small. Indeed, set
\[
\Omega_\varepsilon
:=
\Omega
\setminus
\Bigl(
B(\xi_{1,\varepsilon},r\delta_{1,\varepsilon})
\cup
B(\xi_{2,\varepsilon},r\delta_{2,\varepsilon})
\Bigr).
\]
Suppose by contradiction that there exists a third nodal domain
$\omega_\varepsilon$. Since the two balls above are contained in the two
distinguished nodal domains, necessarily
$\omega_\varepsilon\subset\Omega_\varepsilon$. Up to replacing
$u_\varepsilon$ by $-u_\varepsilon$ on $\omega_\varepsilon$, we may assume
that $u_\varepsilon>0$ there. Then $u_\varepsilon$ solves
\begin{equation}\label{pbae}
\begin{cases}
-\Delta u_\varepsilon
=
a_\varepsilon(x)u_\varepsilon
& \text{in }\omega_\varepsilon,\\
u_\varepsilon>0
& \text{in }\omega_\varepsilon,\\
u_\varepsilon=0
& \text{on }\partial\omega_\varepsilon,
\end{cases}
\end{equation}
where
\[
a_\varepsilon:=|u_\varepsilon|^{2}+\e
\]
We compute
\[
\begin{aligned}
\|a_\e \|_{L^{2}(\Omega_\varepsilon)}
&\lesssim
\|\mathcal U_{\delta_{1,\varepsilon},\xi_{1,\varepsilon}}^{2}
\|_{L^{2}(\Omega_\varepsilon)}
+\|
\mathcal  U_{\delta_{2,\varepsilon},\xi_{2,\varepsilon}}^{2}\|_{L^{2}(\Omega_\varepsilon)}
+
\|
|\phi_\varepsilon|^{2}\|_{L^2(\Omega_\varepsilon)}+\e
\\
&\leq c_1
\left(
\int_{|y|\geq r}
\frac{1}{(1+|y|^2)^4}\,dy
\right)^{1/2}
+
c_2
\|\phi_\varepsilon\|_{L^{4}(\Omega)}^{2}+c_3\e.
\end{aligned}
\]
We choose $r$ so large that
$$
c_1
\left(
\int_{|y|\geq r}
\frac{1}{(1+|y|^2)^4}\,dy
\right)^{1/2}
\leq
\frac{S_4^{-2}}{6},
$$
where $S_4$ is the best Sobolev constant.

Fixing such an $r$, there exists $\varepsilon_0>0$ such that,
for every $\varepsilon\in(0,\varepsilon_0)$,
$$
c_2
\|\phi_\varepsilon\|_{L^{4}(\Omega)}^{2}
\leq
\frac{S_4^{-2}}{6}\quad \mbox{and}\quad c_3\e\leq \frac{S_4^{-2}}{6}.
$$
Hence we get that
\begin{equation}\label{stimaae}
\|a_\e\|_{L^{2}(\Omega_\varepsilon)}<S_4^{-2}.
\end{equation}
From \eqref{pbae} we get that
\[
\|u_\e\|^2
=
\int_{\omega_\e} a_\e (x) u^2_\e\, dx
\leq
\|a_\e \|_{L^{2}}
\|u_\e\|_{L^{4}}^2
\leq
S_4^2
\|a_\e\|_{L^{2}}
\|u_\e\|^2.
\]
Since, by \eqref{stimaae}
\[
S_4^2\|a_\e\|_{L^{N/2}}<1,
\]
we obtain a contradiction. Hence the problem \eqref{pbae} does not have any solution. 
\item[(iii)] Let $u_\varepsilon$ be as in \eqref{4.1} and assume that $\tau_\Omega(\xi_1^*)=\tau_\Omega(\xi_2^*)$. Let $\bd\xi_\e=(\xi_{1, \e}, \xi_{2, \e})$. Then, as $\e\to 0$, $$\Lambda_2(\bd\xi_{\e}) \to \Lambda_2(\bd\xi^*),\quad \mbox{and}\quad \ff_{2, 2}(\bd\xi_\e)\to \ff_{2, 2}(\bd\xi^*)=1. 
$$
We need to understand the behavior of $u_\e$ defined in \eqref{4.1} outside the concentration points.\\
Indeed, first, we remark that, by \eqref{exp1} we have that

$$
u_\e(x)=\mathfrak C e^{-\frac{8\pi^2}{\e}\Lambda_2(\bd\xi_\e)}\left(G(x, \xi_{1, \e})-\ff_{2, 2}(\bd\xi_\e)G(x, \xi_{2, \e})+o(1)\right)$$
and then, as $\e\to 0$,
$$
\frac{1}{e^{-\frac{8\pi^2}{\e}\Lambda_2(\bd\xi_\e)}}u_\e(x)\to \mathfrak C \left(G(x, \xi_{1}^*)-G(x, \xi_{2}^*)\right)$$
pointwise.
Set
\[
q_\varepsilon:=
e^{-\frac{8\pi^2}{\varepsilon}\Lambda_2(\boldsymbol\xi_\varepsilon)}
=\delta_{1,\varepsilon}.
\]
We first record the quantitative estimate
$$
\bigl\||u_\varepsilon|^3+\varepsilon|u_\varepsilon|\bigr\|_{L^1(\Omega)}
\le Cq_\varepsilon .
$$
Indeed, using the decomposition \eqref{4.1}, the estimate
\eqref{stimaphi}, and $\delta_{2,\varepsilon}/q_\varepsilon\to1$, we obtain
\[
\begin{aligned}
\bigl\||u_\varepsilon|^3+\varepsilon|u_\varepsilon|\bigr\|_{L^1(\Omega)}
&\le C\sum_{i=1}^2\int_\Omega
\mathcal U_{\delta_{i,\varepsilon},\xi_{i,\varepsilon}}^3\,dx
+C\int_\Omega|\phi_\varepsilon|^3\,dx\\
&\quad
+C\varepsilon\sum_{i=1}^2\int_\Omega
\mathcal U_{\delta_{i,\varepsilon},\xi_{i,\varepsilon}}\,dx
+C\varepsilon\int_\Omega|\phi_\varepsilon|\,dx\\
&\le C q_\varepsilon.
\end{aligned}
\]

We shall also use the following standard consequence of the
Lyapunov--Schmidt estimate and boundary elliptic regularity:
for every compact set
$K\Subset\overline\Omega\setminus\{\xi_1^*,\xi_2^*\}$,
\begin{equation}\label{C1-away}
\frac{\phi_\varepsilon}{q_\varepsilon}\longrightarrow0
\qquad\text{in }C^1(K).
\end{equation}
For completeness, one may obtain \eqref{C1-away} by applying local (and,
near $\partial\Omega$, boundary) $\varepsilon$--regularity to
$\phi_\varepsilon/q_\varepsilon$. Indeed
$\|\phi_\varepsilon/q_\varepsilon\|_{H_0^1(\Omega)}
=O(q_\varepsilon+\varepsilon)\to0$, while on sets staying a fixed
distance from the concentration points the approximate solution is
$O(q_\varepsilon)$; the equation for the remainder then gives the
required bootstrap.

Combining \eqref{exp1} (in its $C^1$ version away from the pole) with
\eqref{C1-away}, we obtain
\begin{equation}\label{C1-green-limit}
\frac{u_\varepsilon}{q_\varepsilon}
\longrightarrow
\mathfrak C\bigl(G(\cdot,\xi_1^*)-G(\cdot,\xi_2^*)\bigr)
\quad\text{in }
C^1_{\mathrm{loc}}
\bigl(\overline\Omega\setminus\{\xi_1^*,\xi_2^*\}\bigr).
\end{equation}
In particular, the convergence holds uniformly for the normal
derivatives on $\partial\Omega$.

Let, now, $$
g
:=
\mathfrak C \Bigl[
G(\cdot,\xi_1^*)
-
G(\cdot,\xi_2^*)
\Bigr]
$$

Then, as a consequence of the previous discussion,
\[
\frac{1}{e^{-\frac{8\pi^2}{\e}\Lambda_2(\bd\xi_\e)}}
\frac{\partial u_\varepsilon}{\partial\nu}
\to
\frac{\partial g}{\partial\nu}
\qquad
\text{on }\partial\Omega.
\]

Suppose by contradiction that
\[
\overline{\{x\in\Omega:u_\varepsilon(x)=0\}}
\cap\partial\Omega=\emptyset .
\]
Since $\partial\Omega$ is connected, a sufficiently thin tubular
neighbourhood of $\partial\Omega$ inside $\Omega$ is connected. The
assumption above therefore implies that $u_\varepsilon$ has a constant
sign in that neighbourhood. By the Hopf boundary lemma (with $\nu$ the
outward normal),
$\partial_\nu u_\varepsilon$ has a fixed strict sign on
$\partial\Omega$. Passing to the limit in \eqref{C1-green-limit}, we
deduce that $\partial_\nu g$ has one sign on $\partial\Omega$.

On the other hand, since
$-\Delta_xG(x,\xi)=\delta_\xi$ and $G(\cdot,\xi)=0$ on
$\partial\Omega$,
\[
\int_{\partial\Omega}\partial_\nu G(x,\xi)\,d\sigma=-1.
\]
Consequently
\[
\int_{\partial\Omega}\partial_\nu g\,d\sigma=0.
\]
A continuous function with one sign and zero integral must vanish
identically; hence
\[
\partial_\nu G(\cdot,\xi_1^*)
=
\partial_\nu G(\cdot,\xi_2^*)
\qquad\text{on }\partial\Omega.
\]
This is impossible. Indeed, for every harmonic function $h$ in
$\Omega$, Green's representation formula gives
\[
h(\xi)
=
-\int_{\partial\Omega}
h(x)\,\partial_\nu G(x,\xi)\,d\sigma(x).
\]
Thus the preceding equality would imply
$h(\xi_1^*)=h(\xi_2^*)$ for every harmonic $h$. Choosing the coordinate
functions $h(x)=x_j$, $j=1,\ldots,4$, yields
$\xi_1^*=\xi_2^*$, contradicting
$\boldsymbol\xi^*\in\Omega_2^*$.

\end{itemize}

\end{proof}

In the case $N\geq 5$, Theorem \ref{main-two-peak}-(i) was proved in \cite{MiPi} (see also \cite{BMP} for another related problem). 

\section{Symmetric configurations with more than two peaks}
For three or more peaks, the finite-dimensional problem is no longer tractable in a general domain. We therefore impose symmetries that reduce the location and scale parameters to a low-dimensional problem. This section proves the configurations collected in Theorem~\ref{main-configurations}. We treat the configurations in increasing order of complexity: alternating peaks on a regular polygon, two regular polygons in orthogonal planes, one central peak surrounded by a polygon of peaks of the opposite sign, and finally aligned three-, four-, and five-peak configurations. In each case the strategy is the same: identify the relevant symmetry-reduced interaction matrix, select a simple positive eigenvalue with a positive eigenvector, and exhibit a stable critical set for that eigenvalue.

\subsection{Alternating peaks on a regular polygon}

Let $k\ge2$ be even. Suppose that $\Omega\subset\mathbb R^4=\mathbb R^2\times\mathbb R^2$
is invariant under rotations in the first coordinate plane and under the reflections
$x_3\mapsto-x_3$, $x_4\mapsto-x_4$. For $\rho$ in the radial interval of $\Omega$, set
\[
\xi_j(\rho)=\bigl(\rho e^{2\pi\mathtt i(j-1)/k},0\bigr),
\qquad j=1,\ldots,k.
\]
We seek a solution with alternating signs,
\begin{equation}\label{solpalla1}
u_\varepsilon(x)\sim
\sum_{j=1}^k(-1)^{j+1}
P\mathcal U_{\delta,\xi_j(\rho)}(x),
\qquad
\delta=e^{-\frac{8\pi^2\lambda}{\varepsilon}}.
\end{equation}

The natural symmetry is anti-invariance under the rotation
$R_{2\pi/k}$ in the $(x_1,x_2)$-plane, i.e.
 $
u(R_{2\pi/k}x)=-u(x),
$
together with evenness in $x_3$ and $x_4$. Equivalently, the solution is invariant
under $R_{4\pi/k}$. In this symmetry class all concentration scales are equal and
all centers are determined by the single radial parameter $\rho$.

The symmetry-reduced finite-dimensional problem is
\[
\lambda=\Lambda_k(\rho),\qquad \Lambda_k'(\rho)=0,
\]
where
$$
\Lambda_k(\rho)
=
\tau_\Omega(\xi_1(\rho))
-\sum_{j=1}^{k-1}(-1)^j
G(\xi_1(\rho),\xi_{j+1}(\rho)).
$$
Thus it is enough to find a stable critical set of $\Lambda_k$ on which
$\Lambda_k>0$. We verify this first for the ball and then for an annulus.

\subsubsection{The case of the ball}
Let us consider  the ball of radius $R$
\begin{equation*}
B_R = \{ x \in \mathbb{R}^{4} : |x|<R \}.
\end{equation*} 
The Robin function and the Green function in the ball of radius $R$  are explicit and given by
$$
\tau_R(x):=c_4\frac{R^{2}}{(R^2-|x|^2)^{2}},\qquad c_4:=\frac{1}{2\omega}=\frac{1}{4\pi^2}.$$
and 
$$
G_{R}(x, y)=c_4\left(\frac{1}{|x-y|^{2}}-\frac{1}{|(|x|/R)y-(R/|x|)x|^{2}}\right).\\$$ 
In what follows we let $R=1$. Since the common factor $c_4>0$ multiplies both the Green and Robin functions, we suppress it in the dimensionless computations below. Thus every reduced eigenvalue computed in this normalization must be multiplied by $c_4$ when the original Green-function normalization is restored.
Then we get that
$$\begin{aligned}
\Lambda(\rho)&:=\frac{1}{(1-\rho^2)^2}-\sum_{j=1}^{k-1}(-1)^j \left(\frac{1}{|\xi_1(\rho)-\xi_{j+1}(\rho)|^2}-\frac{1}{\left|\rho\xi_{j+1}(\rho)-\frac 1 \rho \xi_1(\rho)\right|^2}\right)\\
&=\frac{1}{(1-\rho^2)^2}-\sum_{j=1}^{k-1}(-1)^j \left(\frac{1}{4\rho^2\sin^2\frac{\pi j}{k}}-\frac{1}{1+\rho^4-2\rho^2\cos\frac{2\pi j}{k}}\right)\end{aligned}$$
We use Lemma \ref{sommasin2} and Lemma \ref{sommaab} with $a=1+\rho^4$ and $b=2\rho^2$ (easily we get that $a>b$). Then, since $a^2-b^2 =(1-\rho^4)^2$ and $a-b=(1-\rho^2)^2$ and $r=\frac{a-\sqrt{a^2-b^2}}{b}=\rho^2$, we get ($k$ is even) 
$$\Lambda(\rho)=\frac{k^2+2}{24\rho^2}+\frac{2k}{1-\rho^4}\frac{\rho^k}{1-\rho^{2k}}.$$ 
It follows immediately that
\[
\Lambda_k(\rho)>0\qquad\text{for every }\rho\in(0,1),
\]
and
\[
\Lambda_k(\rho)\longrightarrow+\infty
\quad\text{as }\rho\to0^+\ \text{or}\ \rho\to1^-.
\]
Hence the set
\[
\mathscr K_k^{\rm ball}:=
\left\{\rho\in(0,1):
\Lambda_k(\rho)=\min_{(0,1)}\Lambda_k\right\}
\]
is nonempty and compact. Since it is the complete set of global minimizers,
it is a strict local minimum set and therefore a stable critical set.

\begin{proposition}\label{prop-alternating-ball}
Let $k\ge2$ be even. For $\varepsilon>0$ sufficiently small, problem
\eqref{pb} in the unit ball admits a sign-changing solution with $k$
concentration points at the vertices of a regular $k$-gon, with alternating
signs as in \eqref{solpalla1}. Along a subsequence,
$\rho_\varepsilon\to\rho_0\in\mathscr K_k^{\rm ball}$ and
\[
\varepsilon\log\delta_\varepsilon^{-1}
\longrightarrow
8\pi^2 c_4\,\Lambda_k(\rho_0),
\]
where $\Lambda_k$ is the dimensionless function displayed above and
$c_4=1/(4\pi^2)$ restores the original Green-function normalization.
\end{proposition}

\subsubsection{The case of the annulus}
Let us consider  the annulus (with $b=1$)
\begin{equation*}
\Omega_{a} = \{ x \in \mathbb{R}^{4} : a<|x|<1 \},\ 0<a<1.
\end{equation*} 
Here the Robin function and the Green function of an annulus are explicitly given by (see \cite{GV}) $$
\tau_{a}(x):=\frac{1}{\omega_3}\sum_{m=0}^{\infty}d_mQ_m(|x|)$$ and 
$$
G_a(x, y):=\frac{1}{2\omega_3|x-y|^2}-\frac{1}{\omega_3}\sum_{m=0}^{\infty} Q_m(x, y) Z_m\left(\frac{x}{|x|}, \frac{y}{|y|}\right).$$

Here $d_m:=(m+1)^2$,
$$
Q_m(x, y):=\frac{a^{2m+2} - a^{2m+2}(|x|^{2m+2}+|y|^{2m+2}) + |x|^{2m+2}|y|^{2m+2}}{(2m+2)(|x||y|)^{m+2}(1-a^{2m+2})}$$
$$
Q_m(|x|):=\frac{a^{2m+2} - 2a^{2m+2}|x|^{2m+2} + |x|^{4m+4}}{(2m+2)|x|^{2m+4}(1-a^{2m+2})}$$
and
$Z_m(\zeta, \eta)$ represents the zonal harmonics of degree $m$ which have a particularly simple expression in terms of the Gegenbauer (or ultraspherical) polynomials $C_{m}^{\lambda}$. The latter can be defined in terms of generating functions. If we write (see \cite{SW} p. 148)
\begin{equation*}
(1-2rt+r^{2})^{-\lambda} = \sum_{m=0}^{\infty} C_m^{\lambda}(t)r^{m},
\end{equation*}
where $0 \leq r<1$, $|t| \leq 1$ and $\lambda>0$, then the coefficient $C_m^{\lambda}$ is called Gegenbauer polynomial of degree $m$ associated with $\lambda$. \ \\
Furthermore, from Theorem 2.1 of \cite{GV}, we have that for any $\zeta$ and $\eta$ such that $|\zeta|=1$ and $|\eta|=1$ it holds
\begin{equation}\label{propgege}
Z_{m}(\zeta,\eta) = (m+1)C_{m}^{1}(\zeta\cdot\eta).
\end{equation}
It is also known that $$C_m^1(1)=(m+1),\quad C_m^1(-x)=(-1)^mC_m^1(x).$$
Moreover $Z_{m}$ is invariant by rotation.\\\\ 
For fixed $a\in(0,1)$, the function $\Lambda_k$ diverges at both components
of the boundary of $(a,1)$:
\[
\Lambda_k(\rho)\to+\infty
\qquad\text{as }\rho\to a^+\ \text{or}\ \rho\to1^-.
\]
Hence its global minimum is attained in $(a,1)$. We shall prove
\begin{equation}\label{c1}
\min_{\rho\in(a,1)}\Lambda_k(\rho)>0.
\end{equation}
In particular,
\begin{equation}\label{lambda0}
\begin{aligned}
\Lambda_k(\rho) & = \frac{1}{\omega_3} \Bigg[ \sum_{m=0}^{+\infty} d_{m} Q_{m}(\rho) - \sum_{j=1}^{k-1} (-1)^{j} \Bigg( \frac{1}{2|\xi_1(\rho)-\xi_{j+1}(\rho)|^{2}} - \sum_{m=0}^{+\infty} Q_{m}(\rho) Z_{m}\left( \frac{\xi_1}{|\xi_1|}, \frac{\xi_{j+1}}{|\xi_{j+1}|}\right) \Bigg) \Bigg].
\end{aligned}
\end{equation}
Furthermore, we observe that, by using Lemma \ref{sommasin2}, 
\begin{equation}\label{stimaI1}
 \sum_{j=1}^{k-1} \frac{(-1)^{j}}{|\xi_1(\rho)-\xi_{j+1}(\rho)|^{2}} = \frac{1}{4\rho^{2}}  \sum_{j=1}^{k-1} \frac{(-1)^{j}}{\sin^2 \left(\frac{\pi j}{k}\right)}=-\frac{k^2+2}{24\rho^2}.
\end{equation}
Now, by using \eqref{stimaI1} and \eqref{propgege}, we can rewrite \eqref{lambda0} as follows
\begin{equation*}
\Lambda_k(\rho) = \frac{1}{\omega_3} \Bigg[ \frac{k^2+2}{48\rho^2} + \sum_{m=0}^{+\infty} (m+1)Q_{m}(\rho) \Bigg( (m+1) + \sum_{j=1}^{k-1} (-1)^{j}C_m^1\left(\underbrace{ \frac{\xi_1}{|\xi_1|}\cdot \frac{\xi_{j+1}}{|\xi_{j+1}|}}_{:=\cos\frac{2\pi j}{k}}\right)\Bigg) \Bigg].\\
\end{equation*}
For $\rho\in(a,1)$ one has
\[
Q_m(\rho)
=
\frac{(\rho^{2m+2}-a^{2m+2})^2+a^{2m+2}(1-a^{2m+2})}
{(2m+2)\rho^{2m+4}(1-a^{2m+2})}>0.
\]
Thus, once the coefficients $\alpha_{k,m}$ below are known to be nonnegative,
the positivity of $\Lambda_k$ follows directly from the strictly positive
singular term.

We let
$$
\alpha_{k,m} = (m+1) +\underbrace{\sum_{j=1}^{k-1} (-1)^{j} C_{m}^{1} \left( \cos\frac{2\pi j}{k}\right)}_{:=S_{k, m}}.  
$$
In order to obtain \eqref{c1}, we prove that $\alpha_{k,m} \geq 0$ for all $m \in \mathbb{N}$.  \ \\
Firstly, we observe that, reasoning as in \cite{MRV} - (3.40),
\begin{equation}\label{stimaimportantemk}
\begin{aligned}
\sum_{j=1}^{k-1} (-1)^{j} \cos \left( m\frac{2 \pi j}{k} \right)
&=\frac12\sum_{j=1}^{k-1}
\left[
\cos\left(\left(m+\frac k2\right)\frac{2\pi j}{k}\right)
+
\cos\left(\left(m-\frac k2\right)\frac{2\pi j}{k}\right)
\right]\\
&=\left\{\begin{aligned} &k-1\quad&\mbox{if}\,\, m+\frac k 2 \,\, \mbox{is a multiple of}\,\, k\\
&-1&\ \ \mbox{if}\,\, m+\frac k 2 \,\, \mbox{is not a multiple of}\,\, k.\end{aligned}\right.
\end{aligned}
\end{equation}
We recall that $C_{0}^{1}(x)=1$ (see \cite{Gra} - 8.930 (1.10)) and then
\begin{equation}\label{alphak0}
\begin{aligned}
\alpha_{k,0} & = 1 + \sum_{j=1}^{k-1}(-1)^j= 0
\end{aligned}
\end{equation}
since 
\begin{equation*}
\sum_{j=1}^{k-1} (-1)^{j} = \frac{1-(-1)^k}{2}-1.
\end{equation*}
Instead, using the fact $C_{1}^{1}(x)=2x$ (see \cite{Gra} - 8.930 (2.10)), we have that
\begin{equation}\label{alphak1}
\begin{aligned}
\alpha_{k,1} & =2+ 2\sum_{j=1}^{k-1} (-1)^{j} \cos \left( \frac{2 \pi j}{k} \right)=\left\{\begin{aligned}&0\,\,&\mbox{ for any}\, k>2\\ &4 \,\,&\mbox{ for }\, k=2\end{aligned}\right.
\end{aligned}
\end{equation}
since by \eqref{stimaimportantemk} with $m=1$ we have that $$\sum_{j=1}^{k-1} (-1)^{j} \cos \left( \frac{2 \pi j}{k} \right)=\left\{\begin{aligned}&-1\,\,&\mbox{ for any}\, k>2\\ &1 \,\,&\mbox{ for }\, k=2.\end{aligned}\right.$$ 
Now, if we prove the following monotonicity property
\begin{equation}\label{claim}
\alpha_{k,m} \geq \alpha_{k,m-2}
\end{equation}
for all $m \geq 2$, using \eqref{alphak0} and \eqref{alphak1}, we can conclude that 
$$
\alpha_{k,m} \geq \left\{
\begin{aligned} 
&\alpha_{k,0} \geq 0 \quad \text{if } m \text{ is even} \\
&\alpha_{k,1} \ge 0 \quad \text{if } m \text{ is odd}
\end{aligned}
\right.
$$
and so, the last step in order to obtain the nonnegativity of $\alpha_{k,m}$ is to prove \eqref{claim}. \ \\
First of all we recall the following recursive formula (see (3.50) of \cite{MRV})
\begin{equation*}
C_{m}^{1}(\cos t) = 2\cos(mt) + C_{m-2}^{1}(\cos t).
\end{equation*}
So, if we multiply by $(-1)^{j}$ and we sum over $j$, we obtain that
\begin{equation*}
S_{k,m} = 2 \sum_{j=1}^{k-1} (-1)^{j} \cos \left(m \frac{2 \pi j}{k} \right) + S_{k,m-2}.
\end{equation*}
Now, we can observe that 
\begin{equation*}
\alpha_{k,m-2} = m-1 + S_{k,m-2}
\end{equation*}
and hence by using again \eqref{stimaimportantemk}
$$
\alpha_{k,m}-\alpha_{k,m-2} = 2 + 2 \sum_{j=1}^{k-1} (-1)^{j} \cos \left(m \frac{2 \pi j}{k} \right)\geq 0;
$$
so,
\begin{equation*}
\Lambda_k(\rho) = \frac{1}{\omega_3} \Bigg[ \frac{k^2+2}{48\rho^2} + \sum_{m=0}^{+\infty} (m+1)Q_{m}(\rho) \alpha_{k,m} \Bigg] > 0
\end{equation*}
and the claim holds.

Define
\[
\mathscr K_{k,a}^{\rm ann}:=
\left\{\rho\in(a,1):
\Lambda_k(\rho)=\min_{(a,1)}\Lambda_k\right\}.
\]
The set $\mathscr K_{k,a}^{\rm ann}$ is nonempty and compact and, being the
complete set of global minimizers, is a stable critical set.

\begin{proposition}\label{prop-alternating-annulus}
Let $a\in(0,1)$ and let $k\ge2$ be even. For $\varepsilon>0$ sufficiently
small, problem \eqref{pb} in the annulus $\Omega_a$ admits a sign-changing
solution with $k$ alternating concentration points at the vertices of a
regular $k$-gon. Along a subsequence,
$\rho_\varepsilon\to\rho_0\in\mathscr K_{k,a}^{\rm ann}$.
\end{proposition}

Moreover the minimum point $\rho_a$ is such that
\begin{equation}\label{condminrhoa}
\frac{k^2+2}{24} = \sum_{m=0}^{+\infty} \Bigg[d_{m} + \sum_{j=1}^{k-1} (-1)^{j} Z_{m} \left( \frac{\xi_1}{|\xi_1|}, \frac{\xi_{j+1}}{|\xi_{j+1}|}\right) \Bigg] Q'_{m}(\rho) \rho^3.
\end{equation}
Therefore, from \eqref{condminrhoa}, we can observe that 
\begin{equation}\label{condminrhoa1}
\frac{k^2+2}{24} \leq \sum_{m=0}^{+\infty} \Bigg[ d_m +\Bigg| \sum_{j=1}^{k-1} (-1)^{j}  Z_{m} \left( \frac{\xi_1}{|\xi_1|}, \frac{\xi_{j+1}}{|\xi_{j+1}|} \right) \Bigg| \Bigg] |Q'_{m}(\rho)| \rho^3  \leq k \sum_{m=0}^{+\infty} d_{m} |Q'_{m}(\rho)| \rho^3
\end{equation}
since (see Proposition $5.27$ of \cite{ABR})
\begin{equation*}
\Bigg| Z_{m} \left( \frac{\xi_1}{|\xi_1|}, \frac{\xi_{j+1}}{|\xi_{j+1}|}\right) \Bigg| \leq d_{m}.
\end{equation*}
Now, recalling that $d_{m}=(m+1)^{2}$ and 
\begin{equation*}
Q'_{m}(\rho)  = \frac{1}{(2m+2)(1-a^{2m+2})} \Bigg[-(2m+4)a^{2m+2} \rho^{-2m-5} + 4a^{2m+2} \rho^{-3} + 2m \rho^{2m-1} \Bigg],
\end{equation*}
we can rewrite \eqref{condminrhoa1} as 
\begin{equation}\label{CONDMIN}
\begin{aligned}
\frac{k^2+2}{24k} & \leq \sum_{m=0}^{+\infty} \frac{m+1}{(1-a^{2m+2})} \Bigg[(m+2)a^{2m+2} \rho^{-2m-2} + 2a^{2m+2} + m \rho^{2m+2} \Bigg]
\\& \leq \left( \frac{1}{1-a^{2}} \right) \sum_{m=0}^{+\infty} (m+1) \Bigg[(m+2)a^{2m+2} \rho^{-2m-2} + 2a^{2m+2} + m \rho^{2m+2} \Bigg]
\\& =  \left( \frac{1}{1-a^{2}} \right) \Bigg[ \frac{2 \rho^{4} a^{2}}{(\rho+a)^{3}} \frac{1}{(\rho-a)^{3}} + \frac{2a^{2}}{(1-a^{2})^{2}} + \frac{2 \rho^{4}}{(1-\rho^2)^{3}} \Bigg].
\end{aligned}
\end{equation}
Consequently, if $\rho_k$ is any sequence of critical points satisfying
\eqref{condminrhoa} (in particular, any sequence of global minimizers), then
\[
\operatorname{dist}(\rho_k,\{a,1\})\longrightarrow0
\qquad\text{as }k\to\infty.
\]
Indeed, the left-hand side of \eqref{CONDMIN} tends to $+\infty$, whereas
the right-hand side stays uniformly bounded on every compact subinterval
of $(a,1)$. We do not need to decide here which boundary component is
selected.

\begin{remark}
We emphasize that the case of alternating-sign peaks in the annulus differs substantially from the case in which all peaks are positive. Indeed, in \cite{MRV} the authors proved that the existence of a solution with all positive peaks is guaranteed provided that the inner hole of the annulus is sufficiently large. By contrast, in the alternating-sign case considered here, no assumption on the size of the hole is required.
\end{remark}
\subsection{Two orthogonal regular polygons in the ball}

Let $\Omega=B_1(0)\subset\mathbb R^4=\mathbb R^2\times\mathbb R^2$ and let
$k\ge1$. We consider two regular $k$-gons lying in the orthogonal coordinate
planes,
\[
\xi_j(\rho)
=
\rho\bigl(e^{2\pi\mathtt i(j-1)/k},0\bigr),
\qquad
\eta_j(\rho)
=
\rho\bigl(0,e^{2\pi\mathtt i(j-1)/k}\bigr),
\qquad j=1,\ldots,k,
\]
with $\rho\in(0,1)$. We seek a solution with positive bubbles at the
$\xi_j$'s and negative bubbles at the $\eta_j$'s.

We work in the symmetry class generated by the rotations of angle $2\pi/k$
in each coordinate plane together with the antisymmetry
\[
u(x',x'')=-u(x'',x').
\]
The latter forces the two polygons to have the same concentration scale.
Thus the ansatz reduces to
\[
u_\varepsilon(x)
\sim
\sum_{j=1}^k P\mathcal U_{\delta,\xi_j(\rho)}(x)
-
\sum_{j=1}^k P\mathcal U_{\delta,\eta_j(\rho)}(x),
\qquad
\delta=e^{-\frac{8\pi^2\lambda}{\varepsilon}}.
\]

In the dimensionless Green-function normalization used throughout this
subsection, the scale equation is
\[
\lambda=\Lambda_k^\perp(\rho),
\]
where
$$
\Lambda_k^\perp(\rho)
=
\lambda_{0,k}(\rho)+k f(\rho),
$$
with
\[
\lambda_{0,k}(\rho)
=
\tau(\rho)
-
\sum_{j=1}^{k-1}G(\xi_1(\rho),\xi_{j+1}(\rho)),
\qquad
f(\rho)
=
G(\xi_1(\rho),\eta_j(\rho))
=
\frac1{2\rho^2}-\frac1{1+\rho^4}.
\]
The quantity $f(\rho)$ is independent of $j$ because the two polygons lie in
orthogonal planes.

Using the explicit Green function of the unit ball and Lemma~\ref{sommaab},
we obtain
$$
\Lambda_k^\perp(\rho)
=
\frac{k(1+\rho^{2k})}
{(1-\rho^4)(1-\rho^{2k})}
-
\frac{k^2-1}{12\rho^2}
+
\frac{k}{2\rho^2}
-
\frac{k}{1+\rho^4}.
$$
If $1\le k\le6$ the function $\Lambda_k^\perp(\rho)$  has a minimizer.
Indeed
\[
\Lambda_k^\perp(\rho)
\sim
\frac{-k^2+6k+1}{12\rho^2}
\qquad\text{as }\rho\to0^+.
\]
Hence
\[
\Lambda_k^\perp(\rho)\to+\infty
\quad\text{as }\rho\to0^+
\qquad\text{if }1\le k\le6.
\]
Moreover,
\[
\Lambda_k^\perp(\rho)\to+\infty
\qquad\text{as }\rho\to1^-.
\]

The positivity of the minimum value follows directly from
\[
\frac{k(1+\rho^{2k})}
{(1-\rho^4)(1-\rho^{2k})}
>
\frac{k}{1+\rho^4},
\]
which yields
\[
\Lambda_k^\perp(\rho)
>
\left(
\frac{k}{2}-\frac{k^2-1}{12}
\right)\frac1{\rho^2}
=
\frac{-k^2+6k+1}{12\rho^2}
>0.
\]

Define
\[
\mathscr K_k^\perp
=
\left\{
\rho\in(0,1):
\Lambda_k^\perp(\rho)
=
\min_{(0,1)}\Lambda_k^\perp
\right\}.
\]
For $1\le k\le6$, the set $\mathscr K_k^\perp$ is nonempty and compact.
Since it is the complete set of global minimizers, it is a strict local
minimum set and therefore stable.
\\

It is worthwhile to remark that  if  $k\geq 7$ the function
$
\Lambda_k^\perp(\rho)
$
is strictly increasing in $(0,1)$.

We first observe that
\[
-\frac{k^2-1}{12\rho^2}
+\frac{k}{2\rho^2}
=
-\frac{k^2-6k-1}{12\rho^2}.
\]
Moreover,
\begin{align*}
\frac{1+\rho^{2k}}
{(1-\rho^4)(1-\rho^{2k})}
-\frac{1}{1+\rho^4}
&=
\frac{(1+\rho^{2k})(1+\rho^4)
-(1-\rho^4)(1-\rho^{2k})}
{(1-\rho^8)(1-\rho^{2k})}
\\
&=
\frac{2(\rho^4+\rho^{2k})}
{(1-\rho^8)(1-\rho^{2k})}.
\end{align*}
Hence,
\[
\Lambda_k^\perp(\rho)
=
2k\frac{\rho^4+\rho^{2k}}
{(1-\rho^8)(1-\rho^{2k})}
-
\frac{k^2-6k-1}{12\rho^2}.
\]

Set
\[
F_k(\rho)
:=
\frac{\rho^4+\rho^{2k}}
{(1-\rho^8)(1-\rho^{2k})}.
\]
Then
\[
(\Lambda_k^\perp)'(\rho)
=
2k F_k'(\rho)
+
\frac{k^2-6k-1}{6\rho^3}.
\]
A direct differentiation gives
\begin{align*}
F_k'(\rho)
={}&
\frac{4\rho^3+2k\rho^{2k-1}}
{(1-\rho^8)(1-\rho^{2k})}
\\
&+
\frac{\rho^4+\rho^{2k}}
{(1-\rho^8)(1-\rho^{2k})}
\left(
\frac{8\rho^7}{1-\rho^8}
+
\frac{2k\rho^{2k-1}}{1-\rho^{2k}}
\right).
\end{align*}
Since $0<\rho<1$, all the terms on the right-hand side
are strictly positive. Therefore,
\[
F_k'(\rho)>0
\qquad\text{for every }\rho\in(0,1).
\]

On the other hand, for $k\geq7$,
\[
k^2-6k-1>0.
\]
Consequently,
\[
(\Lambda_k^\perp)'(\rho)
=
2k F_k'(\rho)
+
\frac{k^2-6k-1}{6\rho^3}
>0
\]
for every $\rho\in(0,1)$ and every $k\geq7$.
Thus $\Lambda_k^\perp$ is strictly increasing in $(0,1)$.
\\

\begin{proposition}\label{prop-orthogonal-polygons}
Let $1\le k\le6$. For $\varepsilon>0$ sufficiently small, problem
\eqref{pb} in the unit ball admits a sign-changing solution with $k$
positive concentration points at the vertices of a regular $k$-gon in one
coordinate plane and $k$ negative concentration points at the vertices of
a regular $k$-gon in the orthogonal coordinate plane. Along a subsequence,
\[
\rho_\varepsilon\to\rho_0\in\mathscr K_k^\perp,
\]
and
\[
\varepsilon\log\delta_\varepsilon^{-1}
\longrightarrow
8\pi^2 c_4\,\Lambda_k^\perp(\rho_0).
\]
\end{proposition}

\subsection{One positive peak and a negative regular polygon}

Let $k\ge2$ and let $B=B_1(0)\subset\mathbb R^4$. We consider a configuration consisting of one positive peak at the origin and $k$ negative peaks at the vertices
\[
\xi_j(\rho)=\rho\bigl(e^{2\pi \mathtt i(j-1)/k},0\bigr)\in\mathbb R^2\times\mathbb R^2,
\qquad j=1,\ldots,k,
\]
of a regular polygon, where $0<\rho<1$. We look for a solution of the form
\[
u_\e=P\mathcal U_{\delta,0}-\sum_{j=1}^kP\mathcal U_{\gamma,\xi_j(\rho)}+\phi_\e,
\]
with
\[
\delta=e^{-8\pi^2\lambda/\e},\qquad
\gamma=e^{-8\pi^2\lambda/\e}d,\qquad d>0.
\]

We perform the Lyapunov--Schmidt reduction in the class invariant under the rotation of angle $2\pi/k$ in the $(x_1,x_2)$-plane and under the full orthogonal group acting on the $(x_3,x_4)$-plane. These symmetries fix the central peak at the origin and force the $k$ outer peaks to have the same scale and to lie on a regular polygon of radius $\rho$. Consequently, after solving the infinite-dimensional equation, the remaining scale equations are the two equations corresponding to the central orbit and to one representative of the polygon orbit.

For simplicity we suppress the common Green-function factor $c_4=1/(4\pi^2)$ throughout this subsection; with the original normalization, the matrix and all its eigenvalues are multiplied by $c_4$. Set
\[
\beta(\rho):=G(0,\xi_1(\rho))=\frac1{\rho^2}-1>0
\]
and
\[
\lambda_0(\rho):=\tau(\xi_1(\rho))-
\sum_{j=2}^kG(\xi_1(\rho),\xi_j(\rho)).
\]
The scale equations are
\begin{equation}\label{scale-1k}
\begin{cases}
1+k\beta(\rho)d=\lambda,\\
\beta(\rho)+\lambda_0(\rho)d=\lambda d.
\end{cases}
\end{equation}
Introducing $x=(1,\sqrt{k}\,d)^T$, system~\eqref{scale-1k} is equivalent to the symmetric eigenvalue problem
\begin{equation}\label{matrix-1k}
B_k(\rho)x=\lambda x,
\qquad
B_k(\rho):=
\begin{pmatrix}
1&\sqrt{k}\,\beta(\rho)\\
\sqrt{k}\,\beta(\rho)&\lambda_0(\rho)
\end{pmatrix}.
\end{equation}
Since $\beta(\rho)>0$, the largest eigenvalue
$$
\Lambda_k(\rho):=
\frac{1+\lambda_0(\rho)+
\sqrt{(1-\lambda_0(\rho))^2+4k\beta(\rho)^2}}{2}
$$
is simple for every $\rho\in(0,1)$, and its eigenvector can be chosen with both components strictly positive. More precisely,
\[
d_k(\rho)=
\frac{\lambda_0(\rho)-1+
\sqrt{(1-\lambda_0(\rho))^2+4k\beta(\rho)^2}}
{2k\beta(\rho)}>0.
\]
Moreover, since the off-diagonal entry in~\eqref{matrix-1k} is nonzero,
\begin{equation}\label{Lambda-1k-positive}
\Lambda_k(\rho)>\max\{1,\lambda_0(\rho)\}\ge1.
\end{equation}

We now study the dependence on $\rho$. By Lemma~\ref{sommaab},
\[
\sum_{j=2}^kG(\xi_1(\rho),\xi_j(\rho))
=
\frac{k^2-1}{12\rho^2}
-\frac{k(1+\rho^{2k})}{(1-\rho^4)(1-\rho^{2k})}
+\frac1{(1-\rho^2)^2}.
\]
Since $\tau(\rho)=(1-\rho^2)^{-2}$, it follows that
\begin{equation}\label{lambda0-1k}
\lambda_0(\rho)
=-\frac{k^2-1}{12\rho^2}
+\frac{k(1+\rho^{2k})}{(1-\rho^4)(1-\rho^{2k})}.
\end{equation}
Therefore, as $\rho\to0^+$,
\[
\lambda_0(\rho)=-\frac{A_k}{\rho^2}+O(1),\qquad
\beta(\rho)=\frac1{\rho^2}+O(1),
\qquad A_k:=\frac{k^2-1}{12},
\]
and hence
\begin{equation}\label{asymptotic-1k-zero}
\Lambda_k(\rho)
=\frac{-A_k+\sqrt{A_k^2+4k}}{2\rho^2}+O(1)
\longrightarrow+\infty.
\end{equation}
Notice that the coefficient is strictly positive for every $k\ge2$. On the other hand, from~\eqref{lambda0-1k},
$\lambda_0(\rho)\to+\infty$ as $\rho\to1^-$, and therefore, by~\eqref{Lambda-1k-positive},
\begin{equation}\label{asymptotic-1k-one}
\Lambda_k(\rho)\longrightarrow+\infty
\qquad\text{as }\rho\to1^-.
\end{equation}

Thus $\Lambda_k$ attains its minimum in $(0,1)$. Let
\[
\mathscr K_k^{\,1+k}:=
\left\{\rho\in(0,1):
\Lambda_k(\rho)=\min_{(0,1)}\Lambda_k\right\}.
\]
By~\eqref{asymptotic-1k-zero}--\eqref{asymptotic-1k-one}, this set is nonempty and compact. Since it is the full set of global minimizers, it is a strict local minimum set and hence a stable critical set in the sense of Definition~\ref{yy1}. At every $\rho_0\in\mathscr K_k^{\,1+k}$, the Hellmann--Feynman formula gives
\[
\Lambda_k'(\rho_0)=x_k(\rho_0)^TB_k'(\rho_0)x_k(\rho_0)=0,
\]
where $x_k$ is the normalized positive eigenvector of $B_k$. This is precisely the remaining radial equation in the symmetry-reduced finite-dimensional problem.

We have therefore proved the following result.
\begin{proposition}\label{one-plus-k-polygon}
For every integer $k\ge2$ and for $\e>0$ sufficiently small, problem~\eqref{pb} in the unit ball of $\mathbb R^4$ admits a sign-changing solution with one positive concentration point at the origin and $k$ negative concentration points located at the vertices of a regular polygon in a two-dimensional plane. More precisely, along a subsequence the polygonal radius satisfies
\[
\rho_\e\longrightarrow\rho_0\in\mathscr K_k^{\,1+k},
\]
the ratio between the common outer scale and the central scale converges to $d_k(\rho_0)>0$, and
\[
\e\log\delta_\e^{-1}
\longrightarrow8\pi^2\Lambda_k(\rho_0)\]
after suppressing the common factor $c_4$;
restoring the original Green-function normalization multiplies $\Lambda_k$ by $c_4$.

\end{proposition}

\begin{remark}
The construction has no analogue of the restriction $k\le6$ occurring in Proposition~\ref{prop-orthogonal-polygons}. Indeed, the coefficient in~\eqref{asymptotic-1k-zero} is positive for every $k$, so the relevant eigenvalue is coercive at $\rho=0$ for arbitrary polygonal multiplicity.
\end{remark}

\subsection{Three aligned peaks in a symmetric domain}

Assume that $\Omega$ contains the origin, is invariant under the reflection
\[
(x_1,x')\longmapsto(-x_1,x'),
\qquad x'\in\mathbb R^3,
\]
and under a group of orthogonal transformations acting on the $x'$ variables
whose fixed-point set is the $x_1$-axis. Assume moreover that the connected
component of $\Omega\cap(\mathbb R e_1)$ containing the origin is
$(-R,R)e_1$. These symmetries allow us to restrict the reduced problem to
configurations with one point at the origin and a symmetric pair at
$\pm re_1$, $r\in(0,R)$.

We prescribe one positive peak at the origin and two negative peaks at
$\pm re_1$. On the symmetry-reduced scale space, the signed interaction
matrix is
\[
M_3(r)=
\begin{pmatrix}
\tau_\Omega(0)&\sqrt2\,G(0,re_1)\\
\sqrt2\,G(0,re_1)&
\tau_\Omega(re_1)-G(re_1,-re_1)
\end{pmatrix}.
\]
Let $\Lambda_3(r)$ denote its largest eigenvalue. Since
$G(0,re_1)>0$, this eigenvalue is simple and its associated eigenvector
can be chosen with both components strictly positive. Moreover
\[
\Lambda_3(r)\ge \tau_\Omega(0)>0.
\]

As $r\to R^-$, the point $re_1$ approaches the boundary and therefore
\[
\tau_\Omega(re_1)\longrightarrow+\infty,
\qquad
\Lambda_3(r)\longrightarrow+\infty.
\]
As $r\to0^+$, using the Green-function normalization of Section~1,
\[
G(0,re_1)=\frac{c_4}{r^2}+O(1),
\qquad
G(re_1,-re_1)=\frac{c_4}{4r^2}+O(1),
\]
while $\tau_\Omega(re_1)=\tau_\Omega(0)+O(r)$. Hence
\[
M_3(r)
=
\frac{c_4}{r^2}
\begin{pmatrix}
0&\sqrt2\\
\sqrt2&-\frac14
\end{pmatrix}
+O(1),
\]
and consequently
\[
\Lambda_3(r)
\sim
\frac{c_4}{r^2}\,
\frac{-\frac14+\sqrt{\frac1{16}+8}}{2}
\longrightarrow+\infty.
\]

Therefore
\[
\mathscr K_3:=
\left\{
r\in(0,R):
\Lambda_3(r)=\min_{(0,R)}\Lambda_3
\right\}
\]
is nonempty and compact. It is the complete set of global minimizers, hence
a stable critical set.

\begin{proposition}\label{prop-three-aligned}
Under the symmetry assumptions above, for $\varepsilon>0$ sufficiently
small problem~\eqref{pb} admits a sign-changing solution with one positive
concentration point at the origin and two negative concentration points at
$\pm r_\varepsilon e_1$, where, along a subsequence,
\[
r_\varepsilon\longrightarrow r_0\in\mathscr K_3.
\]
The two outer bubbles have the same scale, and their common scale relative
to the central one is determined by the positive eigenvector of
$M_3(r_0)$.
\end{proposition}

\subsection{Four aligned peaks in a convex symmetric domain}

Assume in addition that $\Omega$ is convex. We consider the sign pattern
\[
-r_2e_1\;(-),\qquad -r_1e_1\;(+),\qquad
r_1e_1\;(-),\qquad r_2e_1\;(+),
\qquad 0<r_1<r_2<R.
\]
Equivalently, the solution is odd under the reflection
$x_1\mapsto-x_1$. The two symmetry orbits are the pairs
$\{\pm r_1e_1\}$ and $\{\pm r_2e_1\}$; within each orbit the signs are
opposite but the concentration scales are equal.

In the orthonormal orbit basis, the corresponding signed interaction matrix is
\[
M_4(r_1,r_2)=
\begin{pmatrix}
\tau_\Omega(r_1e_1)+G(r_1e_1,-r_1e_1)
&
G(r_1e_1,r_2e_1)-G(-r_1e_1,r_2e_1)
\\[1mm]
G(r_1e_1,r_2e_1)-G(-r_1e_1,r_2e_1)
&
\tau_\Omega(r_2e_1)+G(r_2e_1,-r_2e_1)
\end{pmatrix}.
\]
By Lemma~A.2 of Bartsch--D'Aprile--Pistoia~\cite{BDP2013},
\[
(x-y)\cdot\nabla_xG(x,y)<0
\qquad (x\ne y)
\]
in a convex domain. Taking $y=r_2e_1$ and comparing the two points
$r_1e_1$ and $-r_1e_1$ on the same half-line issuing from $y$, we obtain
\[
G(r_1e_1,r_2e_1)>G(-r_1e_1,r_2e_1).
\]
Thus the off-diagonal entry of $M_4$ is strictly positive. Its largest
eigenvalue, denoted by $\Lambda_4(r_1,r_2)$, is therefore simple and has
an associated eigenvector with strictly positive components. Since the
diagonal entries are positive,
\[
\Lambda_4(r_1,r_2)>0.
\]

Let
\[
\mathcal D_4:=\{(r_1,r_2):0<r_1<r_2<R\}.
\]
We claim that
\[
\Lambda_4(r_1,r_2)\longrightarrow+\infty
\qquad\text{as }(r_1,r_2)\to\partial\mathcal D_4.
\]
Indeed, if $r_2\to R^-$, then
$\tau_\Omega(r_2e_1)\to+\infty$. If $r_1\to0^+$, then
\[
G(r_1e_1,-r_1e_1)
=
\frac{c_4}{4r_1^2}+O(1)\longrightarrow+\infty,
\]
and if $r_2\to0^+$ then necessarily $r_1\to0^+$ as well. Finally, if
$r_2-r_1\to0^+$ while the points stay away from $0$ and $\partial\Omega$,
then
\[
G(r_1e_1,r_2e_1)
=
\frac{c_4}{(r_2-r_1)^2}+O(1),
\]
whereas $G(-r_1e_1,r_2e_1)$ remains bounded; hence the positive
off-diagonal entry tends to $+\infty$, and so does $\Lambda_4$.

It follows that
\[
\mathscr K_4:=
\operatorname*{argmin}_{\mathcal D_4}\Lambda_4
\]
is a nonempty compact stable critical set.

\begin{proposition}\label{prop-four-aligned}
Under the symmetry and convexity assumptions above, for $\varepsilon>0$
sufficiently small problem~\eqref{pb} admits a sign-changing solution with
four aligned concentration points
\[
-r_{2,\varepsilon}e_1,\quad -r_{1,\varepsilon}e_1,\quad
r_{1,\varepsilon}e_1,\quad r_{2,\varepsilon}e_1
\]
and alternating signs $-,+,-,+$. Along a subsequence,
\[
(r_{1,\varepsilon},r_{2,\varepsilon})
\longrightarrow(r_1^0,r_2^0)\in\mathscr K_4.
\]
The two orbit scales are determined by the positive eigenvector of
$M_4(r_1^0,r_2^0)$.
\end{proposition}

\subsection{Five aligned peaks in the ball}
We look for a solution of the form
\[
u_\e=P\mathcal U_{\delta,0}-P\mathcal U_{\gamma,\xi_1}+P\mathcal U_{\mu,\xi_2}
-P\mathcal U_{\gamma,-\xi_1}+P\mathcal U_{\mu,-\xi_2}+\phi_\e,
\]
where
\[
\delta=e^{-\frac{8\pi^2}{\e}\lambda}d_1,
\qquad \gamma=e^{-\frac{8\pi^2}{\e}\lambda}d_2,
\qquad \mu=e^{-\frac{8\pi^2}{\e}\lambda}d_3,
\]
and
\[
\xi_1=(s,0,0,0),\qquad \xi_2=(t,0,0,0),\qquad 0<s<t<1.
\]
We set
\[
g(r):=\frac1{r^2}-1,
\qquad
\tau(r):=\frac1{(1-r^2)^2},
\qquad
h(r):=\frac1{4r^2}-\frac1{(1+r^2)^2},
\]
and
\[
a(r):=\tau(r)-h(r),
\]
while
\[
F(s,t):=f_-(s,t)+f_+(s,t),
\]
with
\[
f_-(s,t)=\frac1{(t-s)^2}-\frac1{(1-st)^2},
\qquad
f_+(s,t)=\frac1{(s+t)^2}-\frac1{(1+st)^2}.
\]
The reduced equations associated with the three symmetry orbits are
\begin{equation}\label{five-raw-system}
\left\{
\begin{aligned}
&d_1+2g(s)d_2-2g(t)d_3=\lambda d_1,\\
&g(s)d_1+a(s)d_2+F(s,t)d_3=\lambda d_2,\\
&-g(t)d_1+F(s,t)d_2+a(t)d_3=\lambda d_3.
\end{aligned}
\right.
\end{equation}
The first orbit has cardinality one, whereas the other two have cardinality two. Therefore the correct symmetric form of \eqref{five-raw-system} is obtained by introducing
\[
x=(x_1,x_2,x_3)^T:=(d_1,\sqrt2\,d_2,\sqrt2\,d_3)^T.
\]
Then \eqref{five-raw-system} is equivalent to
$$
S(s,t)x=\lambda x,
$$
where
$$
S(s,t):=
\begin{pmatrix}
1&\sqrt2\,g(s)&-\sqrt2\,g(t)\\
\sqrt2\,g(s)&a(s)&F(s,t)\\
-\sqrt2\,g(t)&F(s,t)&a(t)
\end{pmatrix},
\qquad (s,t)\in\mathcal D:=\{0<s<t<1\}.
$$
Thus it is enough to find a stable critical set of the largest eigenvalue
\[
\Lambda(s,t):=\lambda_{\max}(S(s,t))
\]
for which the associated eigenvector has strictly positive components.

\medskip
\noindent\textit{Step 1: coercivity and existence of a minimum.}
We claim that
\[
\Lambda(s,t)\longrightarrow+\infty
\qquad\text{as }(s,t)\to\partial\mathcal D.
\]
Indeed, if $s\to0$, then by the Rayleigh principle
\[
\Lambda(s,t)\geq
\lambda_{\max}
\begin{pmatrix}
1&\sqrt2\,g(s)\\
\sqrt2\,g(s)&a(s)
\end{pmatrix}
\longrightarrow+\infty,
\]
because $g(s)=s^{-2}-1$ and $a(s)=-\frac1{4s^2}+O(1)$. If $t\to1$, then
\[
\Lambda(s,t)\geq a(t)\longrightarrow+\infty.
\]
Finally, if $s$ stays away from $0$, $t$ stays away from $1$, and $t-s\to0$, then $a(s)$ and $a(t)$ remain bounded while $F(s,t)\to+\infty$; testing the Rayleigh quotient with $(0,1,1)/\sqrt2$ gives
\[
\Lambda(s,t)\geq \frac{a(s)+a(t)}2+F(s,t)\longrightarrow+\infty.
\]
Hence the set
\[
K:=\operatorname*{argmin}_{(s,t)\in\mathcal D}\Lambda(s,t)
\]
is nonempty and compactly contained in $\mathcal D$. Moreover $\Lambda\geq1$, by testing with $e_1$, so its minimum is positive.

\medskip
\noindent\textit{Step 2: simplicity of the largest eigenvalue.}
Consider the principal block
\[
B_s:=
\begin{pmatrix}
1&\sqrt2\,g(s)\\
\sqrt2\,g(s)&a(s)
\end{pmatrix}
\]
and denote its largest eigenvalue by $\mu_s$. Let $(\beta_1,\beta_2)$ be the corresponding unit eigenvector, chosen with positive components, and set
\[
\gamma_s:=\frac{\beta_2}{\beta_1}
=\frac{\mu_s-1}{\sqrt2\,g(s)}.
\]
We first note that
\begin{equation}\label{gamma-lower}
\gamma_s>\frac1{\sqrt2}
\qquad\text{for every }s\in(0,1).
\end{equation}
Indeed, writing
\[
Y(s):=\frac{a(s)-1}{2\sqrt2\,g(s)},
\qquad
\gamma_s=Y(s)+\sqrt{Y(s)^2+1},
\]
the inequality \eqref{gamma-lower} is equivalent to
\[
Y(s)>-\frac1{2\sqrt2},
\]
that is, to $a(s)+g(s)>1$. But
\[
a(s)+g(s)-1
=
\frac1{(1-s^2)^2}+\frac1{(1+s^2)^2}+\frac3{4s^2}-2>0.
\]
We shall also use 
\begin{equation}\label{F-lower}
F(s,t)>2g(t),\qquad 0<s<t<1.
\end{equation}
Indeed,  
\[
\frac1{(t-s)^2}+\frac1{(t+s)^2}
=\frac2{t^2}\sum_{m=0}^\infty(2m+1)\left(\frac{s}{t}\right)^{2m},
\]
whereas
\[
\frac1{(1-st)^2}+\frac1{(1+st)^2}
=2\sum_{m=0}^\infty(2m+1)(st)^{2m},
\]
and subtraction gives
\[
F(s,t)=2g(t)+2\sum_{m=1}^\infty(2m+1)s^{2m}
\bigl(t^{-2m-2}-t^{2m}\bigr)>2g(t).
\]
Combining \eqref{gamma-lower} and \eqref{F-lower},
\begin{equation}\label{cross-positive}
F(s,t)\gamma_s-\sqrt2\,g(t)>0.
\end{equation}
Embed $(\beta_1,\beta_2)$ in $\mathbb R^3$ and take
$w_\eta=(\beta_1,\beta_2,\eta)$. Then
\[
\frac{w_\eta^TS(s,t)w_\eta}{|w_\eta|^2}
=
\mu_s+
\frac{2\eta\beta_1\bigl(F(s,t)\gamma_s-\sqrt2\,g(t)\bigr)
+(a(t)-\mu_s)\eta^2}{1+\eta^2}.
\]
By \eqref{cross-positive}, this is strictly larger than $\mu_s$ for all sufficiently small $\eta>0$. Therefore
\begin{equation}\label{strict-interlacing-five}
\Lambda(s,t)>\mu_s.
\end{equation}
Cauchy's interlacing theorem now implies that the largest eigenvalue of $S(s,t)$ is simple throughout $\mathcal D$.

\medskip
\noindent\textit{Step 3: the eigenvector at a minimum has no vanishing component.}
Fix $(s_0,t_0)\in K$ and let $x=(x_1,x_2,x_3)$ be a unit eigenvector associated with $\Lambda(s_0,t_0)$. From \eqref{strict-interlacing-five}, $x_3\neq0$.
Suppose that $x_2=0$. Then $(x_1,x_3)$ is an eigenvector of
\[
B_t:=
\begin{pmatrix}
1&-\sqrt2\,g(t_0)\\
-\sqrt2\,g(t_0)&a(t_0)
\end{pmatrix}
\]
corresponding to its largest eigenvalue $\mu_t$. Up to sign we may write
\[
(x_1,x_3)=c(1,-\gamma_t),
\qquad
\gamma_t:=\frac{\mu_t-1}{\sqrt2\,g(t_0)}>\frac1{\sqrt2}.
\]
The second row of the full eigenvalue equation gives
\[
F(s_0,t_0)\gamma_t=\sqrt2\,g(s_0).
\]
Hence
\[
\mu_t=1+\sqrt2\,g(t_0)\gamma_t
=1+\frac{2g(s_0)g(t_0)}{F(s_0,t_0)}.
\]
On the other hand
\[
\mu_{s_0}=1+\sqrt2\,g(s_0)\gamma_{s_0}>1+g(s_0),
\]
whereas \eqref{F-lower} gives
\[
\mu_t<1+g(s_0).
\]
Thus $\mu_t<\mu_{s_0}\leq\Lambda(s_0,t_0)$, contradicting
$\mu_t=\Lambda(s_0,t_0)$. Therefore
\begin{equation}\label{nonzero-five}
x_2x_3\neq0.
\end{equation}

\medskip
\noindent\textit{Step 4: positivity of the eigenvector at the minimum.}
Since $\Lambda$ is simple, it is smooth near $(s_0,t_0)$ and the Hellmann--Feynman identities give
\[
x^TS_s(s_0,t_0)x=0,
\qquad
x^TS_t(s_0,t_0)x=0.
\]
Using \eqref{nonzero-five}, we may divide by $x_2$ and $x_3$ and obtain
\begin{equation}\label{HF-five}
\left\{
\begin{aligned}
&2\sqrt2\,g'(s_0)x_1+a'(s_0)x_2+2F_s(s_0,t_0)x_3=0,\\
&-2\sqrt2\,g'(t_0)x_1+2F_t(s_0,t_0)x_2+a'(t_0)x_3=0.
\end{aligned}
\right.
\end{equation}
Solving \eqref{HF-five} for $x_1,x_2$ in terms of $x_3$ yields
\[
x_1=
\frac{\sqrt2\bigl(a'(s_0)a'(t_0)-4F_sF_t\bigr)}
{4\Delta}\,x_3,
\qquad
x_2=
\frac{-g'(s_0)a'(t_0)-2g'(t_0)F_s}
{\Delta}\,x_3,
\]
where
\[
\Delta:=g'(t_0)a'(s_0)+2g'(s_0)F_t(s_0,t_0).
\]
Here $g'<0$, $a'>0$, $F_s>0$ and $F_t<0$. It remains only to show that $\Delta>0$. Since $g'(r)=-2r^{-3}$, this is equivalent to
\[
2t_0^3\bigl(-F_t(s_0,t_0)\bigr)>s_0^3a'(s_0).
\]
Now
\[
-F_t(s,t)
=
\frac2{(t-s)^3}+\frac{2s}{(1-st)^3}
+\frac2{(s+t)^3}-\frac{2s}{(1+st)^3}
>\frac2{(t-s)^3},
\]
so
\[
2t_0^3(-F_t)>4\left(\frac{t_0}{t_0-s_0}\right)^3
>\frac4{(1-s_0)^3}.
\]
Furthermore
\[
s^3a'(s)=\frac{4s^4}{(1-s^2)^3}+\frac12-
\frac{4s^4}{(1+s^2)^3}<\frac4{(1-s)^3}.
\]
Consequently $\Delta>0$. After normalizing $x_3>0$, the preceding formulas give
\[
x_1>0,\qquad x_2>0,\qquad x_3>0.
\]
Since $d_1=x_1$, $d_2=x_2/\sqrt2$ and $d_3=x_3/\sqrt2$, the three concentration parameters are positive as well.

Finally, $K$ is a strict local minimum set for $\Lambda$: it is compactly contained in $\mathcal D$, $\Lambda$ is constant on $K$, and $\Lambda(z)>\min_{\mathcal D}\Lambda$ for every $z\notin K$. Hence, by Definition~\ref{yy1}, $K$ is a stable critical set. By simplicity, the normalized eigenvector depends continuously on $(s,t)$ and its three components remain positive in a neighbourhood of $K$. Therefore Theorem~\ref{main} applies to the five-peak configuration.

\begin{proposition}\label{prop-five-aligned}
For $\varepsilon>0$ sufficiently small, problem~\eqref{pb} in the unit ball
admits a sign-changing solution with five aligned concentration points
\[
0,\qquad \pm s_\varepsilon e_1,\qquad \pm t_\varepsilon e_1,
\qquad 0<s_\varepsilon<t_\varepsilon<1,
\]
and sign pattern
\[
+,-,+,-,+.
\]
Along a subsequence,
\[
(s_\varepsilon,t_\varepsilon)\longrightarrow(s_0,t_0)\in K,
\]
where $K=\operatorname*{argmin}_{\mathcal D}\Lambda$. The three orbit
scales are asymptotic to the strictly positive components of the
eigenvector of $S(s_0,t_0)$ associated with $\Lambda(s_0,t_0)$.
\end{proposition}

\appendix
\section{Auxiliary results}\label{app:auxiliary}
Here we collect some useful results.\\ The first one is useful to find a zero of a $C^1-$ map.

\begin{lemma}\label{isolated}

Let $\mathcal A\times \mathcal B$ be an open subset of $\mathbb R^m\times \mathbb R^h$ and define the $C^1-$ maps $$\Phi_1: \mathcal A\times \mathcal B \to \mathbb R^h,\quad \quad \Phi_2: \mathcal A\times \mathcal B\to \mathbb R^m.$$ Let, also, for every $(x, y)\in\mathcal A\times\mathcal B$ $$\Phi(x, y)=(\Phi_1(x, y), \Phi_2(x, y)).$$Assume for any $x\in\mathcal A$ there exists a unique $y:=y(x)\in\mathcal B$ such that $\Phi_1(x, y(x))=0$ and $D_y \Phi_1(x, y(x))$ is invertible. Let $\varphi(x):=\Phi_2(x, y(x)).$\\ 
If $\mathscr K\subset\mathcal A$ is such that 
$$\varphi(x_0)=0\ \hbox{for any}\ x_0\in\mathscr K\ \hbox{and}\ \mathtt{deg}\left(\varphi, \Theta,0\right)\neq 0$$
for some open neighbourhood $\Theta$ of $\mathscr K$,  then the
set
$$\mathscr Z:=\{(x_0, y_0)\in \mathcal A\times\mathcal B\ :\ x_0\in\mathscr K,\ y_0=y(x_0)\}$$ is such that 
$$\Phi(x_0, y_0)=0\ \hbox{for any}\ (x_0, y_0)\in \mathscr Z\ \hbox{and}\ \mathtt{deg}\left(\Phi, \Xi,0\right)\neq 0,$$
for some open neighbourhood $\Xi$ of $\mathscr Z.$
\end{lemma}
Here we have some Lemmas useful to prove Theorem~\ref{main-two-peak}. The first one was proved in \cite{H}- Lemma 2.

\begin{lemma}\label{lemmastimaGT}
Let $u$ be a solution of
\[
\begin{cases}
-\Delta u=f
& \text{in }\Omega,\\
u=0
& \text{on }\partial\Omega.
\end{cases}
\]
Let $\omega$ be a neighbourhood of $\partial\Omega$.
Then
\[
\|\nabla u\|_{C^{0,\alpha}(\omega^*)}
\leq
C
\Bigl(
\|f\|_{L^1(\Omega)}
+
\|f\|_{L^\infty(\omega)}
\Bigr),
\]
where $\omega^*\Subset\omega$.
\end{lemma}
The second result follows by a regularity result for critical problems in Brezis and Kato \cite{BK} 
\begin{lemma}\label{BK}
Let $w\in H^1_0(\Omega)$ solve
\[
-\Delta w=a(x)w
\]
with $a\in L^{N/2}(\Omega)$ for $N\geq 3$.
Then $w\in L^q(\Omega)$ for any $q<\infty$.
\end{lemma}
\begin{remark}\label{remarkBK}
We apply Lemma \ref{BK} to problem \eqref{pb} in dimension $N=4$ with $a(x):=
|u|^{2}+\e$. Then $a\in L^2(\Omega)$. Then $w\in L^q(\Omega)$ for any $q<\infty$. Hence $f:=u^3+\e|u|\in L^p(\Omega)$ for any $p<\infty$. We use a classical inequalities getting that $$\|u\|_{W^{2, p}(\Omega)}\leq C\| f\|_{L^p(\Omega)}$$ and hence, if we choose $p>\frac N 2$ then $u\in L^{\infty}(\Omega)$.\end{remark}

\section{Matrix and summation identities}\label{app:matrix}
We use an overlooked characterization of interlacing to give a two sentence proof
of Cauchy's interlace theorem. Recall that if polynomials $f(x)$ and $g(x)$ have all
real roots
$$
r_1 \leq r_2 \leq \ldots \leq r_n
\quad\mbox{and}\quad
s_1 \leq s_2 \leq \ldots \leq s_{n-1},
$$
then we say that $f$ and $g$ interlace if and only if
$$
r_1 \leq s_1 \leq r_2 \leq s_2 \leq \cdots \leq s_{n-1} \leq r_n.
$$
The following can be found in \cite{RahmanSchmeisser}, along with a discussion of
its history back to Hermite.

\begin{lemma}\label{positive-2x2}
Let
\[
A=\begin{pmatrix}a&b\\ b&c\end{pmatrix},
\qquad a>0,\quad b>0.
\]
Then its largest eigenvalue
\[
\mu_+(A)=\frac{a+c+\sqrt{(a-c)^2+4b^2}}{2}
\]
is simple and positive. An associated eigenvector is
\[
\begin{pmatrix} b\\ \mu_+(A)-a\end{pmatrix},
\]
and both components are strictly positive.
\end{lemma}

\begin{theorem}\label{cauchy}
The roots of polynomials $f$ and $g$ interlace if and only if the linear
combinations
$
f+\alpha g
$
have all real roots for all $\alpha\in\mathbb{R}$.
\end{theorem}
\begin{corollary}
If $A$ is a Hermitian matrix, and $B$ is a principal submatrix of $A$,
then the eigenvalues of $B$ interlace the eigenvalues of $A$.
\end{corollary}

See \cite{F}.
\begin{theorem}[Decomposition of Weyl]\label{Weyl}
Let $A$ and $B$ two Hermitian matrices $n\times n$ with eigenvalues $\lambda_1\leq \ldots\leq \lambda_n$ and $\mu_1\leq \ldots\leq \mu_n$ respectively.
Let $\gamma_1\leq\ldots\leq\gamma_n$ the eigenvalues of $A+B$. Then 
$$\lambda_j+\mu_{k-j+1}\leq \gamma_k\leq \lambda_i+\mu_{n-i+k},\quad\forall\,\, k\leq n$$
for $1\leq j\leq k\leq i\leq n.$
\\
In particular, it follows that $$
\lambda_n+\mu_1\leq \gamma_n\leq \lambda_n+\mu_n.$$
\end{theorem}

\begin{lemma}\label{sommasin2}
It holds $$ \mathtt S_k:=\sum_{j=1}^{k-1}(-1)^j\frac{1}{\sin^2\frac{\pi j}{k}}=\left\{\begin{aligned} &0 \quad&\hbox{if}\,\, k\,\, \hbox{is odd}\\
&-\frac{k^2+2}{6}\quad&\hbox{if}\,\, k\,\, \hbox{is even}\end{aligned}\right.$$
\end{lemma}
\begin{proof}
\textbf{Case 1: $k$ odd.}\\
We evaluate
$$\begin{aligned} \mathtt S_k&=\sum_{j=1}^{\frac{k-1}{2}}(-1)^j\frac{1}{\sin^2\frac{\pi j}{k}}+\sum_{j=\frac{k-1}{2}+1}^{k-1}(-1)^j\frac{1}{\sin^2\frac{\pi j}{k}}\\
&=\sum_{j=1}^{\frac{k-1}{2}}(-1)^j\frac{1}{\sin^2\frac{\pi j}{k}}+(-1)^k\sum_{j=1}^{\frac{k-1}{2}}
(-1)^j\frac{1}{\sin^2\frac{\pi (k-j)}{k}}=0\end{aligned}$$ since $k$ is odd and $\sin\left(\frac{\pi j}{k}\right)=\sin\left(\frac{\pi(k-j)}{k}\right).$\\
\medskip
\textbf{Case 2: $k$ even.}\\
We let $k=2m$. We have that by \cite{AZ} and \cite{Fi} (Identity (23)) (see also \cite{MRV}) we have that
\beq\label{stimapos}
\sum_{j=1}^{k-1}\frac{1}{\sin^2\!\left(\frac{\pi j}{k}\right)}
=
\frac{k^2-1}{3}.
\eeq
Then

\[
\mathtt E=
\sum_{j=1\atop\\ j\text{ even}}^{k-1}
\frac{1}{\sin^2\!\left(\frac{\pi j}{k}\right)},
\qquad
\mathtt O=
\sum_{j=1\atop \\ j \text{odd}}^{k-1}
\frac{1}{\sin^2\!\left(\frac{\pi j}{k}\right)},\qquad \mathtt S_k=\mathtt E-\mathtt O.
\]

By using \eqref{stimapos}
\[
\mathtt E
=
\sum_{r=1}^{m-1}
\frac{1}{\sin^2\!\left(\frac{\pi r}{m}\right)}
=
\frac{m^2-1}{3}.
\]

Since

\[
\mathtt E+\mathtt O=\sum_{j=1}^{k-1}\frac{1}{\sin^2\!\left(\frac{\pi j}{k}\right)}=\frac{k^2-1}{3}
=
\frac{4m^2-1}{3},
\]

we get

\[
\mathtt O
=
\frac{4m^2-1}{3}
-
\frac{m^2-1}{3}
=
m^2.
\]

Hence

\[
\mathtt S_k
=
\mathtt E-\mathtt O
=
\frac{m^2-1}{3}-m^2
=
-\frac{2m^2+1}{3}.
\]

Since $k=2m$ we get that

\[
\mathtt S_k
=
-\frac{k^2+2}{6}.
\]
\end{proof}

\begin{lemma}\label{sommaab}
It holds, for $a>|b|$ and for $k\geq 2$
\beq\label{sommasenzameno1}\sum_{j=1}^{k-1}\frac{1}{a-b\cos\frac{2\pi j}{k}}=\frac{k}{\sqrt{a^2-b^2}}
\frac{1+r^k}{1-r^k}-\frac{1}{a-b}\eeq and

\beq\label{sommaconmeno1}\sum_{j=1}^{k-1}\frac{(-1)^j}{a-b\cos\frac{2\pi j}{k}}=\left\{\begin{aligned}&\frac{2k}{\sqrt{a^2-b^2}}
\frac{r^{\frac k 2}}{1-r^k}-\frac{1}{a-b}\quad &\hbox{if}\,\, k\,\, \hbox{is even}\\
& 0\quad &\hbox{if}\,\, k\,\, \hbox{is odd}\end{aligned}\right.
\eeq

where
\[
r=\frac{a-\sqrt{a^2-b^2}}{b},
\qquad |r|<1.
\]
\end{lemma}
\begin{proof}
We show first \eqref{sommasenzameno1}.\\
Let $z_j:=e^{\frac{2\pi \mathtt i j}{k}}$ and $\alpha_j:=\frac{2j \pi }{k} $. Then, $z_j^k=1$.\\ We get
$$\frac{1}{a-b\cos\alpha_j}=\frac{1}{a-\frac b 2(z_j+z_j^{-1})}=-\frac{2z_j}{bz_j^2-2az_j+b}=-\frac{2z_j}{b(z_j-r_+)(z_j-r_-)}$$ where
$$r_\pm :=\frac{a\pm\sqrt{a^2-b^2}}{b}.$$
Then, we can decompose, obtaining 
$$\frac{1}{a-b\cos\alpha_j}=-\frac{2}{b(r_+-r_-)}\left(\frac{r_+}{z_j-r_+}-\frac{r_-}{z_j-r_-}\right)$$ 
In \cite{SteinShakarchi} it is shown that
$$\sum_{j=0}^{k-1}\frac{1}{z_j-r}=\frac{k r^{k-1}}{1-r^k}.$$ Then
$$\begin{aligned} \sum_{j=0}^{k-1}\frac{1}{a-b\cos\frac{2\pi j}{k}}&=-\frac{2}{b(r_+-r_-)}\sum_{j=0}^{k-1}\left(\frac{r_+}{z_j-r_+}-\frac{r_-}{z_j-r_-}\right)=\frac{2k}{b(r_+-r_-)}\left(-\frac{r_+^k}{1-r_+^k}+\frac{r_-^k}{1-r_-^k}\right)\\
&=\frac{2k}{b(r_+-r_-)}\frac{1+r_-^k}{1-r_-^{k}}\end{aligned}$$
where we have used the fact that $r_-r_+=1$. Since $r_+-r_-=\frac{2}{b}\sqrt{a^2-b^2}$ the thesis follows.\\\\
Now we prove \eqref{sommaconmeno1}.\\
{\bf Case $k$ odd.}\\  Then, if we denote by $\mathcal S_k:=\sum_{j=1}^{k-1}\frac{(-1)^j}{a-b\cos\frac{2\pi j}{k}}$, we get
$$\begin{aligned}\mathcal S_k&=\sum_{j=1}^{\frac{k-1}{2}}\frac{(-1)^j}{a-b\cos\frac{2\pi j}{k}}+\sum_{j=\frac{k-1}{2}+1}^{k-1}\frac{(-1)^j}{a-b\cos\frac{2\pi j}{k}}= \sum_{j=1}^{\frac{k-1}{2}}\frac{(-1)^j}{a-b\cos\frac{2\pi j}{k}}+\sum_{\ell=1}^{\frac{k-1}{2}}\frac{(-1)^{k-\ell}}{a-b\cos\frac{2\pi (\ell-k)}{k}}\\
&=\sum_{j=1}^{\frac{k-1}{2}}\frac{(-1)^j}{a-b\cos\frac{2\pi j}{k}}-\sum_{j=1}^{\frac{k-1}{2}}\frac{(-1)^j}{a-b\cos\frac{2\pi j}{k}}=0.\end{aligned}$$
{\bf Case $k$ even.}\\ Let now $k$ even (i.e. $k=2n$). Then
$$\begin{aligned}S_{2n}&=\sum_{j=0}^{2n-1}\frac{(-1)^j}{a-b\cos\frac{\pi j}{n}}=\sum_{j\, \small{even}}\frac{(-1)^j}{a-b\cos\frac{\pi j}{n}}+\sum_{j\, \small{odd}}\frac{(-1)^j}{a-b\cos\frac{\pi j}{n}}\\
&=\underbrace{\sum_{m=0}^{n-1}\frac{1}{a-b\cos\frac{2m\pi}{n}}}_{:=\mathcal E}-\underbrace{\sum_{m=0}^{n-1}\frac{1}{a-b\cos\frac{\pi (2m+1)}{n}}}_{:=\mathcal O}.\end{aligned}$$ 
By using \eqref{sommasenzameno1} we get that $$\mathcal E:=\frac{n}{\sqrt{a^2-b^2}}\frac{1+r^n}{1-r^n}.$$ Moreover $$\mathcal E+\mathcal O=\frac{2n}{\sqrt{a^2-b^2}}\frac{1+r^{2n}}{1-r^{2n}}.$$ Hence $$\mathcal O=\frac{n}{\sqrt{a^2-b^2}}\frac{1-r^n}{1+r^n}.$$ Finally $$S_{2n}:=\frac{n}{\sqrt{a^2-b^2}}\frac{4r^n}{1-r^{2n}}$$ and the thesis follows since $n=\frac k 2$.
\end{proof}

\section*{Data Availability Statements}
All data generated or analysed during this study are included in this article.
\section*{Declarations}
{\bf Conflicts of Interest:} The authors declare they have no financial interests.

\end{document}